\documentclass{amsart}

\usepackage{amssymb}
\usepackage{amsthm}
\usepackage{amsrefs}
\usepackage{mathtools}
\usepackage[hidelinks]{hyperref}
\usepackage[super]{nth}
\usepackage{cancel}
\usepackage{tikz-cd}
\usetikzlibrary{decorations.pathmorphing}
\usepackage{microtype}

\theoremstyle{plain}
\newtheorem{theorem}{Theorem}[section]
\newtheorem{lemma}[theorem]{Lemma}
\newtheorem{proposition}[theorem]{Proposition}
\newtheorem{corollary}[theorem]{Corollary}
\newtheorem{question}[theorem]{Question}

\theoremstyle{definition}
\newtheorem{definition}[theorem]{Definition}
\newtheorem{example}[theorem]{Example}

\theoremstyle{remark}
\newtheorem{remark}[theorem]{Remark}

\numberwithin{equation}{section}

\DeclareMathOperator{\R}{\mathbb{R}}

\DeclareMathOperator{\E}{\mathbb{E}}
\DeclareMathOperator{\GL}{GL}
\DeclareMathOperator{\SL}{SL}
\DeclareMathOperator{\Hom}{Hom}
\DeclareMathOperator{\Gr}{Gr}
\DeclareMathOperator{\Sym}{S}
\DeclareMathOperator{\tr}{tr}

\DeclareMathOperator{\I}{\mathcal{I}}
\DeclareMathOperator{\K}{\mathcal{K}}
\DeclareMathOperator{\Id}{I}
\DeclareMathOperator{\contact}{\mathcal{C}}

\newcommand{\coframeBundle}{\mathcal{F}}
\renewcommand{\d}{\ensuremath{\hspace{2pt} d}}
\newcommand{\dd}[1]{\mathop{d#1}}

\DeclareMathOperator{\SO}{SO}
\newcommand{\cinfinity}[1]{C^\infty\left(#1\right)}

\renewcommand{\mod}[1]{\hspace{4mm}\left(\mbox{mod } #1\right)}

\newcommand{\pDeriv}[2]{\frac{\partial #1}{\partial #2}}
\newcommand{\trans}[1]{{}^t\! #1}

\newcommand{\w}{{\mathchoice{\,{\scriptstyle\wedge}\,}{{\scriptstyle\wedge}}
      {{\scriptscriptstyle\wedge}}{{\scriptscriptstyle\wedge}}}}

\newcommand{\gBun}{\mathcal{B}}

\newcommand{\thetas}[1]{\theta_{#1}}
\newcommand{\thetasu}[1]{\underline{\theta}_{#1}}
\newcommand{\thetanotu}{\thetasu{\varnothing}}
\newcommand{\pis}[1]{\pi_{#1}}
\newcommand{\pisu}[1]{\underline{\pi}_{#1}}
\newcommand{\etu}{\underline{\eta}}
\newcommand{\omegu}{\underline{\omega}}

\newcommand{\omegasu}[1]{\underline{\omega}_{(#1)}}
\newcommand{\newcoframe}[1]{\widetilde{#1}}

\newcommand{\goursam}{\hat{a}}
\newcommand{\sbforms}{\Omega_{sb}}

\newcommand{\cartTheta}{\mathcal{J}}

\newcommand{\jets}{J^2}

\newcommand{\bianchi}{\mathfrak{b}}

\begin{document}

\title{Geometry and Conservation Laws for a Class of Second-Order Parabolic Equations I: Geometry}

\author{Benjamin B. McMillan}
\address{Math Department, University of Adelaide}
\email{benjamin.mcmillan@adelaide.edu.au}
\thanks{This material is based upon work supported by the National Science Foundation under Grant No. 74341.2010}

\date{}

\pagestyle{plain}

\begin{abstract}
I consider the geometry of the general class of scalar \nth{2}-order differential equations with parabolic symbol, including non-linear and non-evolutionary parabolic equations. After defining the appropriate $G$-structure to model parabolic equations, I apply Cartan techniques to determine local geometric invariants (quantities invariant up to a generalized change of variables). One family of invariants gives a geometric characterization for parabolic equation of Monge-Amp\`ere type. A second family of invariants determines when a parabolic equation has a local choice of coordinates putting it in evolutionary form.

In addition to their intrinsic interest, these results are applied in a follow up paper on the conservation laws of parabolic equations. It is shown there that conservation laws for any evolutionary parabolic equation depend on at most second derivatives of solutions. As a corollary, the only evolutionary parabolic equations with at least one non-trivial conservation law are of Monge-Amp\`ere type.
\end{abstract}

\maketitle

\section*{Introduction}
\subsection{Geometry}
The goal of this paper is to study the geometry of differential equations of parabolic type. I define \emph{parabolic systems}---exterior differential systems that are (locally) equivalent to parabolic scalar PDE---in arbitrarily many variables, and make progress on their equivalence problem. The intention is to further our understanding of the geometry of  parabolic equations, so as to provide insight into their structure. For example, in a follow up to this paper, the geometric understanding developed here is applied to show some quite general results on the conservation laws of parabolic equations.

The primary tool used here is Cartan's method of equivalence, a general approach to solving the equivalence problem for geometric structures. If one is handed a class of geometric structure on real-analytic manifolds (eg. Riemannian, almost complex, CR), the equivalence method prescribes how to classify the local invariants of the geometry, which are special functions that distinguish non-isomorphic structures (eg. Riemannian curvature, the Nijenhuis tensor, the Levi form, respectively). Two good references, full of examples, are Gardner's \emph{Method of Equivalence and its Applications} \cite{Gardner:MethodOfEquivalence} and Bryant, Griffiths and Grossman's excellent \emph{Exterior Differential Systems and Euler-Lagrange Partial Differential Equations} \cite{BryantGriffithsGrossman:EDSandEulerLagrangePDEs}. It is worth noting that the primary tool used to classify these invariants is Cartan-K\"ahler Theory, which depends on real analyticity. However, the real-analytic invariants so developed describe smooth invariants as well.

As always in Cartan problems, there is a Lie group associated to the geometry of parabolic equations. The Lie algebra of this group controls the invariants, via the cohomology of the Spencer complex associated to the Lie algebra. A secondary tool used here is the real representation theory of this group, and in particular of its subgroup $\SO(n)$, where $n+1$ is the number of independent variables in the given parabolic equation. The representation theory serves to organize and simplify, allowing for proofs in arbitrarily many dimensions where otherwise the calculations would get out of hand.

To treat differential equations as geometric objects, I consider them as exterior differential systems.
\begin{definition}
  An \emph{exterior differential system} $(M,\I)$ is a smooth manifold $M$ and a graded, differentially closed ideal $\I$ in $\Omega^*(M)$.

  A submanifold $\iota\colon \Sigma \hookrightarrow M$ is an \emph{integral manifold} of $(M,\I)$ if the pullback $\iota^*\I$ is identically zero, or equivalently, if $\phi|_{T_x\Sigma} = 0$ for all $\phi \in \I$ and $x\in \Sigma$.

  An \emph{integral element} at a point $x \in M$ is a subspace $E \subset T_xM$ for which $\phi|_E = 0$ for all $\phi \in \I$.
\end{definition}

Every sufficiently non-degenerate system of partial differential equations corresponds naturally to an exterior differential system. Furthermore, under this correspondence, the graphs of solutions to a PDE are naturally identified with integral submanifolds of the associated EDS.
Very briefly, to a sufficiently non-degenerate differential equation
of order $k$, on $\R^s$-valued functions $u$ of $n$ independent variables $x^a$,
\begin{equation}\label{eq:general PDE}
  F\left(x^a,u,\pDeriv{u}{x^a},\pDeriv{^2 u}{x^a\partial x^b},\ldots\right) = 0 ,
\end{equation}
associate the submanifold
\[ M = F^{-1}(0) \subset J^k(\R^n,\R^s) , \]
where is $F$ considered as a function on the jet space $J^k(\R^n,\R^s)$. The jet space has a canonical Pfaffian ideal $\mathcal{C}$, and denoting by $\I$ the pullback of $\mathcal{C}$ to $M$, the exterior differential system $(M,\I)$ has integral submanifolds that correspond to the graphs of solutions to $F$. To be precise, the $n$-dimensional integral submanifolds that submerse onto $\R^n$ are locally the graphs of solutions. This process is explained in some detail for parabolic equations in Example \ref{ex:jet formulation of parabolic} below.

Exterior differential systems form a category, where a morphism from $(M,\I)$ to $(M',\I')$ is a smooth map $f\colon M \to M'$ that pulls back $\I'$ to a subset of $\I$. Within this category, an integral manifold $\Sigma$ of $(M,\I)$ is simply an EDS-embedding $\varphi\colon (\Sigma,\{0\}) \to (M,\I)$. It is occasionally useful to drop the condition that integral manifolds are embeddings. Then morphisms are characterized by the condition that they push forward solutions of $(M,\I)$ to solutions of $(M',\I')$. The isomorphisms in this category are exactly the maps that preserve the structure of solutions.

\begin{definition}\label{def:eds equivalence}
  An \emph{equivalence} of exterior differential systems $(M,\I)$ and $(M',\I')$ is a diffeomorphism  $f\colon M \to M'$ for which $f^*\I' = \I$.
\end{definition}

One class of example of EDS equivalences is given by the \emph{point transformations}---equivalences induced by changes of variables of a PDE. More precisely, a change of coordinates
\[ x^i,u^a \xrightarrow{\qquad} \tilde x^i(x^i) , \tilde u^a(x^i,u^a) \]
transforms $F$ as in equation \eqref{eq:general PDE} into a new equation $\tilde F$. Then the following diagram is functorial:
\begin{equation*}
  \begin{tikzcd}[row sep = large, column sep=23mm]
    F(x^i,u^a,\pDeriv{u^a}{x^i},\pDeriv{^2 u^a}{x^i \partial x^j},\ldots) = 0 \ar[r, decorate, decoration={zigzag, amplitude=.9}, "\text{EDS functor}"]\arrow{d}[description]{\text{Change of Variables}} & (M,\I) \arrow{d}[description]{\text{EDS equivalence}} \\
    \tilde F(\tilde x^i,\tilde u^a,\pDeriv{\tilde{u^a}}{\tilde{x}^i},\pDeriv{^2 \tilde{u^a}}{\tilde{x}^i \partial \tilde{x}^j},\ldots) = 0 \ar[r, decorate, decoration={zigzag, amplitude=.9}, "\text{EDS functor}"] & (\tilde M, \tilde \I)
  \end{tikzcd}
\end{equation*}
A change of variables on the left takes any solution of $F$ to one of $\tilde F$ and the corresponding EDS equivalence maps the graph of the solution to $F$ to the graph of the solution of $\tilde F$.

Not all EDS equivalences come from point transformations. For example, consider the map from $J^1(\R^n,\R) \cong \R^n \times \R \times \R^n$ to itself given in coordinates by
\[ \varphi(x^i,u,p_i) = (p_i, x^i p_i - u, x^i) . \]
It is straightforward to check that $\varphi$ pulls back the contact ideal $\mathcal{C} = \{ \dd{u} - p_i \dd{x^i} \}$ to itself, and so induces an EDS automorphism of $(J^1(\R^n,\R),\mathcal{C})$. As a consequence, $\varphi$ induces a morphism from any \nth{1}-order PDE to a new exterior differential system in such a way that solutions are taken to solutions. For example, $\varphi$ takes the EDS induced by the equation $\pDeriv{u}{x^i}\pDeriv{u}{x^i} = x_1$ to the EDS induced by the equation $\pDeriv{u}{x^1} = x^i x^i$. Solutions to the second equation can be pushed forward by $\varphi^{-1}$ to give solutions to the first. Since this transformation switches position and derivative variables, it cannot come from any point transformation.

This example also demonstrates that linearity of PDE is not preserved by EDS transformations. In the category of exterior differential systems, an equation can at best be said to be \emph{linearizable}. Note that this is not a defect of EDS equivalences: the same holds for more classical changes of variable, such as point transformations.

Now consider the following.
\begin{question}\label{que:When are PDE equivalent?}
  When are two partial differential equations $F$ and $\tilde F$ related by a change of variables?
\end{question}
The correspondence between PDEs and EDSs suggests the following more geometric question, whose answer essentially answers Question \ref{que:When are PDE equivalent?}.
\begin{question}\label{que:When are EDS equiv?}
  When are two given exterior differential systems related by an EDS equivalence?
\end{question}

The ideal $\I$ is extra geometric structure on $M$ that encodes the structure of solutions, so study of the geometry of $(M,\I)$ provides insight into this question. This is where the method of equivalence comes in, providing invariants that can distinguish between exterior differential systems. But knowledge of the geometry of a differential equation also informs our understanding of the solutions of the differential equation. Bryant, Griffiths and Hsu give a good general overview of this philosophy in the monograph \emph{Toward a Geometry of Differential Equations}, \cite{BryantGriffithsHsu:TowardAGeometryofDifferentialEquations}.

The symbol of a differential equation is a well known example of this. For example, in studying the geometry of the EDS $(M,\I)$ associated to a (non-linear) second-order differential equation, the first invariant one finds is the geometric principal symbol, which can locally be given by a symmetric-matrix valued function on $M$. The signature of this symbol-matrix controls the behaviour of solutions, as is well known for the classical division of linear second-order equations into elliptic, hyperbolic and parabolic classes.

In this paper I provide some answers to Question \ref{que:When are EDS equiv?} in the specific case of systems arising from scalar parabolic equations. I begin by studying the most general class of scalar \nth{2}-order equations that could fairly be called parabolic: the class of PDE for $\R$-valued indeterminate $u$ given by a single equation of the form
\begin{equation}\label{eq:2nd order PDE F first}
  F\left(x^a,u,\pDeriv{u}{x^a},\pDeriv{^2 u}{x^a\partial x^b}\right) = 0 \qquad{a, b = 0,\ldots n}
\end{equation}
whose \emph{geometric principal symbol} is positive semi-definite with 1-dimensional kernel. In more concrete terms, a PDE such as Equation \ref{eq:2nd order PDE F first}  is \emph{weakly parabolic} if its linearization at any $2$-jet of a solution is parabolic in the classical sense. If a parabolic equation is furthermore \emph{evolutionary}, i.e. of the form
\begin{equation}\label{eq:evolution eq}
  \pDeriv{u}{x^0} = F\left(x^0,x^i,u,\pDeriv{u}{x^i},\pDeriv{^2 u}{x^i\partial x^j}\right)  \qquad{i,j = 1,\ldots n},
\end{equation}
then it is \emph{strongly parabolic} or \emph{evolutionary parabolic}.

In Section \ref{sec:parabolic systems}, I define \emph{parabolic systems}, which are the class of exterior differential systems  that model weakly parabolic equations. These look pointwise like the exterior differential system modelling the heat equation
\[ \pDeriv{u}{x^0} = \sum_{i=1}^n \pDeriv{^2u}{x^i \partial x^i} , \]
but are typically more `curved.'

The generic weakly parabolic equation is neither linear, nor evolutionary, for any choice of coordinates, even locally, and the definition of a parabolic system reflects this generality. This generality is a geometrically natural place to start, and then the classical cases are picked out in a geometric manner, depending on the values local invariants introduced here.

Examples of parabolic systems include those arising from evolutionary parabolic equations, both linear and non-linear. Another class of examples is given by scalar parabolic geometric flows, such as the mean curvature flow. The mean curvature flow as a parabolic system is worked out in detail in Example \ref{ex:Mean Curvature Flow}.
Another interesting family of examples is the parabolic Monge-Amp\`ere equations. In fact, the invariants developed here show that if one takes seriously the idea of exploring the geometry of \nth{2}-order PDE, then one is naturally led to Monge-Amp\`ere equations as a geometrically privileged subclass.

In Section \ref{sec:Equivalence setup}, I set up the equivalence problem for parabolic systems. This involves defining the $G$-structure associated to a given parabolic system, which is a  principle $G$-bundle whose sections are the coframings that are adapted to the geometry of the parabolic system. On this bundle, there are tautological forms, the derivatives of which have torsion. This torsion contains obstructions to finding a geometry preserving map to the flat model. But by the usual yoga of equivalence problems, Cartan's technique of the graph upgrades these obstructions to local geometric invariants. As always, it is an interesting problem to pick out the geometrically interesting pieces of this torsion, and to understand what they say of the geometry.

In Section \ref{sec:MA invariants}, I begin tackling the equivalence problem. The first invariants one finds are the \emph{Monge-Amp\`ere} invariants.
These have been introduced for parabolic equations in 2 and 3 variables by Bryant \& Griffiths (\cite{Characteristic_Cohomology_II}) and Clelland (\cite{Clelland_Thesis}) respectively.
In more than 3 dimensions, the representation theory begins to pay off. Indeed, the Monge-Amp\`ere invariants take values in certain $\SO(n)$-representations, which decompose into irreducible components. The different irreducible components have different geometric interpretations. Broadly, the Monge-Amp\`ere invariants can be organized into 3 families, the \emph{primary}, \emph{secondary}, and \emph{tertiary} Monge-Amp\`ere invariants.

In Section \ref{sec:Monge-Ampere Systems}, I recall the class of exterior differential systems that model Monge-Amp\`ere equations and develop their geometry enough to prove Theorem \ref{thm:deprolongation iff invariants vanish iff deprolongation form}, which characterizes a given parabolic equation as Monge-Amp\`ere  if and only if certain irreducible components of its Monge-Amp\`ere invariants vanish. This gives an effective test, which is straightforward in practice.

Many of the well known examples of Monge-Amp\`ere equations have varying symbol type, but are treated as elliptic equations. This is typically achieved by adding a convexity condition to restrict attention to solutions at which the equation has elliptic linearization. However, the geometric considerations here lead to special class of Monge-Amp\`ere equations whose symbol type is defined independently of solutions.
These are the \emph{linear-type Monge-Amp\`ere systems}, which are exactly the Monge-Amp\`ere systems whose symbol depends on at most $1$-jets of solutions (as opposed to on $2$-jets for generic Monge-Amp\`eres). These correspond to a geometrically natural subclass of Monge-Amp\`ere equations, and have a correspondingly geometric characterization: all of their Monge-Amp\`ere invariants vanish. The mean curvature flow gives a good example of a linear type parabolic Monge-Amp\`ere equation.

Linear type Monge-Amp\`ere equations should be of interest even in the theory of elliptic and hyperbolic equations. They contain a natural class of non-linear \nth{2}-order equations which nonetheless have symbol defined independently of solutions.
Furthermore, while the linearity of a PDE is not invariant under changes of variables, any linearizable Monge-Amp\`ere equation is manifestly of linear type. Consequently, linear type Monge-Amp\`ere equations may be considered as a relatively small invariant family of equations containing the linearizable \nth{2}-order equations.

In Section \ref{sec:Goursat invts}, I introduce another family of invariants that arise naturally in parabolic equations, the \emph{extended Goursat} invariants. These are so named because they contain the invariant introduced by Goursat to measure the sub-principal symbol of parabolic equations. The rest of the extended Goursat invariants are new, as they can only be non-zero for parabolic equations in 3 or more variables. Geometrically, they provide an obstruction to the possibility of finding coordinates putting a given parabolic equation in evolutionary form. Indeed, it is shown in Theorem \ref{thm:integrable characteristics} that the primary Goursat invariants of a parabolic equation vanish if and only if it can be put in evolutionary form.



\subsection{Conservation Laws}

The results of this paper are used in a follow up paper, \cite{McMillan:ParabolicsII}, where I prove several general results on the conservation laws of parabolic scalar equations.

For general scalar parabolic equations, the tools of \cite{Characteristic_Cohomology_II} allow me to compute the linear approximation of the characteristic cohomology. This provides a first approximation to the space of conservation laws and also helps define the auxilliary differential equation whose solutions are in bijection with conservation laws.

It is not easy to solve the auxilliary equation for a generic parabolic equation, but the situation is much better for evolutionary parabolic equations.
 The proof of the following theorem relies on the fact that the extended Goursat invariants vanish for an evolutionary parablic system.
\begin{theorem}\label{thm: conservation laws 2nd order}
    Any conservation law for any evolutionary parabolic equation depends on at most second derivatives of solutions.
\end{theorem}
Analogues of this Theorem were proved in $2$ variables by Bryant \& Griffiths \cite{Characteristic_Cohomology_II} and in $3$ variables by Clelland \cite{Clelland_Thesis}.
The Theorem shows that there can be no KdV type phenomenon for evolutionary parabolic equations---there are no hierarchies of conservation laws depending on increasing numbers of derivatives of solutions. An interesting phenomenon in its own right, this Theorem means that the problem of classifying conservation laws for evolutionary parabolic equations is far more tractable than for general PDE.

From Theorem \ref{thm: conservation laws 2nd order}, it is a quick corollary that the existence of even a single non-trivial conservation law puts strong constraint on the geometry:
\begin{corollary}\label{thm: conservation law implies Monge Ampere}
    If an evolutionary parabolic equation has a non-trivial conservation law, then in any neighborhood where the conservation law is not zero, the equation is of Monge-Amp\`ere type.
\end{corollary}
This theorem was shown in 2 and 3 variables by Bryant \& Griffiths and Clelland (respectively), but now the statement is shown in all dimensions.

\section{Parabolic Systems}\label{sec:parabolic systems}

\subsection{Definition  and Examples}
The main object studied in this paper are parabolic systems, which are the exterior differential systems associated to \nth{2}-order scalar parabolic differential equations.
\begin{definition}\label{def:weakly parabolic system}
  A \emph{(weakly) parabolic system} in $n+1$ variables is a \[2n+2+(n+1)(n+2)/2\] dimensional\footnote{This is 1 less than the dimension of $J^2(\R^{n+1},\R)$, corresponding to the fact that a parabolic equation is defined by a single differential relation on $J^2(\R^{n+1},\R)$.} exterior differential system $(M,\I)$ such that any point has a neighborhood equipped with a spanning set of 1-forms
  \begin{equation}\label{eq:0-adapted coframing}
    \makebox[\widthof{$\qquad\qquad a,b = 0,\ldots n$}]{} \thetasu{\varnothing},\thetasu{a},\omegu^a,\pisu{ab} = \pisu{ba} \qquad\qquad a,b = 0,\ldots n
  \end{equation}
  that satisfy:
  \begin{enumerate}
    \item
    The forms $\thetasu{\varnothing},\thetasu{a}$ generate $\I$ as a differential ideal.
    \item
    The structure equations
    \begin{align*}
      \d\thetasu{\varnothing} & \equiv \sum_{a=0}^n -\thetasu{a}\w\omegu^a \mod{\thetasu{\varnothing}} \\
      \d\thetasu{a} & \equiv \sum_{b=0}^n -\pisu{ab}\w\omegu^b \mod{\thetasu{\varnothing},\thetasu{b}} \qquad a, b = 0, \ldots, n .
    \end{align*}
    \item
    The parabolic symbol relation
    \[ \sum_{i=1}^n \pisu{ii} \equiv 0 \mod{\thetasu{\varnothing},\thetasu{a},\omegu^a} . \]
  \end{enumerate}
  Any collection of 1-forms is called a \emph{parabolic (extended) coframing}.
\end{definition}

\begin{remark}
  These exterior differential systems model scalar, parabolic, \nth{2}-order PDE. The first two conditions exhibit $(M,\I)$ as locally equivalent to a \nth{2}-order differential equation. The third condition is then that the principal symbol of this differential equation is everywhere parabolic.

  The existence (or not) of a parabolic-adapted coframing near a point of $M$ is determined entirely by the structure of the ideal $\I$. If they exist, parabolic coframings are the ones adapted to the geometry of the parabolic system $(M,\I)$.

  Parabolic systems are Pfaffian systems, see for example \cite{BCGGG}, chapter IV. Several of the results there are invaluable in the following.
\end{remark}

\begin{remark}
    It will be useful throughout to make a distinction between space-time indices $a, b, c, \ldots$ that range from $0$ to $n$ and spatial indices $i, j, k, l, \ldots$ that range from $1$ to $n$. This convention will strictly held to.

    Because all of the vector spaces that will be associated to a parabolic system have a well defined \emph{spatial trace}, it will be convenient to apply the \emph{spatial Einstein summation} convention, so that, for example,
    \[ \pisu{ii} := \sum_{i=1}^n \pisu{ii} . \]
    The distinction between trace and spatial trace will be made by using (respectively) repeated space-time indices and repeated spatial indices.
\end{remark}

\begin{example}[Jet formulation of a parabolic equation]\label{ex:jet formulation of parabolic}
  Consider the 2-jet bundle $\jets = J^2(\R^{n+1},\R)$ over $\R^{n+1}$, with jet-coordinates $x^a$, $u$, $p_a$, and $p_{ab}=p_{ba}$, where the $p_a$ correspond to the first derivatives of $u$ with respect to $x^a$ and $p_{ab}$ to the second derivatives. These coordinates may be used to define the \emph{contact forms}\footnote{These are not contact forms in the sense of contact geometry, in which a contact form defines a totally non-integrable hyperplane distribution. However, the concepts are related. In particular, the form $\hat\theta_\varnothing$ can be defined on the space of 1-jets, where it \emph{does} define a maximally non-integrable distribution.}
  \begin{align*}
    \hat\theta_\varnothing & = \d u - p_a\d x^a \\
    \hat\theta_a & = \d p_a - p_{ab}\d x^b  ,
  \end{align*}
  as well as the canonical \emph{contact ideal}
  \[ \contact = \{ \hat\theta_\varnothing , \hat\theta_a \}_{EDS} =\{ \hat\theta_\varnothing , \hat\theta_a, \d\hat\theta_\varnothing, \d\hat\theta_a \}_{alg}  \]
  on $\jets$.
  The pair $(\jets,\contact)$ is an exterior differential system\footnote{The geometric structure of this EDS is independent of the choice of coordinates on $\jets$ made above. In fact, $\contact$ can be defined intrinsically: Let $C_x \subset T_x(\jets)$ be the subspace spanned by the tangent planes of all $2$-jet graphs which pass through $x \in \jets$. This defines the $n+n(n+1)/2$ dimmensional \emph{contact distribution} $C$ on $\jets$, and $\contact$ is the differential ideal generated by $C^\perp \subset \Omega^1(M)$.}.

  An $(n+1)$-dimensional submanifold $\Sigma$ of $\jets$ is locally the  2-jet graph of a function $u$ if and only if
  \begin{enumerate}
    \item the `independence condition' $\d x^0 \w\ldots\w \d x^n$ is nonzero when pulled back to $\Sigma$ and
    \item $\Sigma$ is an integral manifold of $(\jets,\contact)$
  \end{enumerate}
  This fact allows one to represent any second order differential equation as an exterior differential system $(M, \I)$. A given non-degenerate \nth{2} order PDE
  \begin{equation}\label{eq:2nd order PDE F}
    F\left(x^a,u,\pDeriv{u}{x^a},\pDeriv{^2 u}{x^a\partial x^b}\right) = 0 \qquad{a,b = 0,\ldots n}
  \end{equation}
  defines a function on $\jets$, and $M$ is the zero locus
  \[ M = F^{-1}(0) . \]
  Then $\I$ is the pullback of $\contact$ to $M$. By construction, integral manifolds of $(M,\I)$ for which
  \[ \d x^0 \w\ldots\w \d x^n|_{\Sigma} \neq 0 \]
  are in local bijection with the 2-jet graphs of solutions to $F$.

  If $F$ is a parabolic equation, then the EDS $(M,\I)$ is a parabolic system. This can be seen by first noting that the spanning set of $1$-forms\footnote{Omitting pullbacks from the notation, which will be done without further comment.}
  \begin{equation}\label{eq:flat coframe}
    \hat\theta_{\varnothing}, \quad \hat\theta_a, \quad \hat\omega^a := \d x^a, \quad \hat\pi_{ab} := \d p_{ab}
  \end{equation}
  of $M$ is partially adapted to $\I$, in that
  \[ \I = \{ \hat\theta_{\varnothing}, \hat\theta_a \} , \]
  and conditions 1 and 2 of Definition \ref{def:weakly parabolic system} hold.
  Now, for any $\GL(\R^{n+1})$-valued function $(B^a_b)$ on $M$, the coframing
  \begin{align*}
    \thetasu{\varnothing} & = \hat\theta_{\varnothing} &
    \thetasu{a} & = B^b_a\hat\theta_b \\
    \omegu^a & = (B^{-1})^a_b \hat\omega^b  &
    \pisu{ab} & = B^c_a \hat\pi_{cd} B^d_b
  \end{align*}
  will also satisfy conditions 1 and 2. On the other hand, there is a non-trivial relation
  \[ 0 = \d F \equiv \pDeriv{F}{p_{ab}}\hat{\pi}_{ab} \mod{\thetasu{\varnothing}, \thetasu{a}, \omegu^a} , \]
  on $M$. By definition, $F$ is parabolic when the `symbol' matrix $\left(\pDeriv{F}{p_{ab}}\right)$ is positive semi-definite. In this case, an appropriate choice for $(B^a_b)$ will diagonalize the symbol as in condition 3 of Definition \ref{def:weakly parabolic system}, so that
  \[ \pisu{ii} \equiv 0 \mod{\thetasu{\varnothing},\thetasu{a},\omegu^a} .  \]

  Any small enough neighborhood of a parabolic system is EDS equivalent to one equipped with such an embedding into $\jets$. See \cite{BCGGG} Theorem 5.10 for example, or the proof of Theorem \ref{thm:integrable characteristics} for details. However, there are parabolic systems which don't have a global embedding into $\jets$, such as the Mean Curvature Flow, taken up in Example \ref{ex:Mean Curvature Flow}.
\end{example}

\begin{example}[Heat equation]\label{ex:heat eq}
  Consider the heat equation
  \[ \pDeriv{u}{x^0} = \sum_{i=1}^n \pDeriv{^2 u}{x^i\partial x^i} \]
  and the corresponding exterior differential system $M$, given by the submanifold $\{ p_0 = p_{ii} \}$ of $J^2(\R^{n+1},\R)$.
  The coframing given in Equation \ref{eq:flat coframe} manifestly
  restricts to a parabolic coframing of $M$.

  Note that the degenerate form of this equation,
  \begin{equation*}
    0 = \sum_{i=1}^n \pDeriv{^2 u}{x^i\partial x^i},
  \end{equation*}
  has parabolic principal symbol, so also defines a parabolic coframing. However, the geometry detects the vanishing sub-principal symbol of this equation, as explained in Section \ref{sec:Goursat invts}.
\end{example}

\begin{example}[Mean Curvature Flow]\label{ex:Mean Curvature Flow}
  Fix an $n$-dimensional manifold $N$. A family
  \[ u \colon N \times \R \to \E^{n+1} \]
  of immersions of $N$ into Euclidean space, paramaterized by $t \in \R$, satisfies \emph{mean curvature flow} if for each $t$ and each point $x\in N$,
  \[ \pDeriv{u}{t}(x,t) \cdot \hat n = H, \]
  where $H$ is the mean curvature of $N \times \{t\}$ at $u(x,t)$ and $\hat n$ the unit normal. I now describe an exterior differential system whose integral manifolds are equivalent to solutions of the mean curvature flow. This system is equivalant to a parabolic system after certain symmetries (technically: Cauchy characteristic directions) are taken into account.

  Let
  \[ M = \R\times\coframeBundle(\E^{n+1})\times \Sym^2(\R^{n+1}) , \]
  where $\coframeBundle(\E^{n+1})$ is the bundle of orthonormal coframes of $\E^{n+1}$ and $\Sym^2(\R^{n+1})$ is the symmetric square. Denote by $t$ the projection of $M$ onto the first factor and $h = (h_{ab})$ the projection onto the last factor. Further, let
  \[ H = \sum_{i=1}^n h_{ii} \]
  denote the spatial trace of $h$. Since $\E^{n+1}$ is a space form, $\coframeBundle(\E^{n+1})$ may be identified with the Lie group $\R^{n+1} \rtimes SO(n+1)$ of Euclidean motions. The components of the Maurer-Cartan form on $SO(n+1)$ extend the tautological form $\eta = (\eta^a)$ of $\coframeBundle(\E^{n+1})$ to a unique $\R^{n+1} \rtimes \mathfrak{so}(n+1)$-valued coframing $\eta^a, \eta^{ab} = - \eta^{ba}$ so that
  \[ \d\eta^a = -\eta^{ab}\w\eta^b \qquad \d\eta^{ab} = -\eta^{ac}\w\eta^{cb} . \]

  These can be used to define the following 1-forms on $M$:
  \begin{equation}\label{eq:MCF coframe}
    \begin{aligned}
      & \thetasu{\varnothing} = \eta^0 - H dt & & \\
      & \thetasu{0} = \d H - h_{0i}\eta^i - h_{00}\d t & & \thetasu{i} = \eta^{0i} - h_{ij}\eta^j - h_{0i}\d t \\
      & \omegu^0 = \d t & & \omegu^i = \eta^i \\
      & \pisu{ij} = \d h_{ij} - h_{ik}\eta^{kj} - h_{jk}\eta^{ki} \\
      & \pisu{0i} = \d h_{0i} - h_{0j}\eta^{ji} - Hh_{ij}h_{jk}\eta^k \\
      & \pisu{00} = \d h_{00} - H h_{0i} h_{ij}\eta^j
    \end{aligned}
  \end{equation}
Let $\I$ be the ideal generated by $\thetasu{\varnothing}, \thetasu{a}$.

  Consider a solution manifold $\Sigma$ of the exterior differential system $(M,\I)$ for which the form $\omegu^0\w\ldots\w\omegu^n$ does not vanish. In particular, $\omegu^0$ is not zero on $\Sigma$, so each point is contained in a small neighborhood foliated by diffeomorphic level surfaces of $t$. Without loss of generality then, suppose that
  \[ \Sigma = (-1,1) \times N \subset M , \]
  and let
  \[ \Sigma_t :=  \{t\}\times N \cong N \]
  for $t \in (-1,1)$.

  Both $\thetasu{\varnothing}$ and $\d t$ vanish on a given $\Sigma_t$, so the form $\eta^0$ pulls back to zero on $\Sigma_t$ as well. Therefore, for each $x \in \Sigma_t$, the tautological coframing $\eta^0,\ldots,\eta^n$ of $\E^{n+1}$ is adapted to $\Sigma_t$, in that
  \[ \left. \eta^0\right|_{T_x\Sigma_t} = 0 \quad \mbox{ and } \quad \eta^0(\hat{n}) = 1 . \]

  Since
  \[ 0 = \d\eta^0 = -\eta^{0i}\w\eta^i \]
  on $\Sigma_t$, Cartan's lemma implies that there is a $\Sym^2(\R^{n})$-valued function $(\tilde h_{ij})$ so that
  \[ \eta^{0i} = \tilde h_{ij}\eta^j . \]
  The function $(\tilde h_{ij})$ is the second fundamental form of $\Sigma_t$ in the orthonormal coframing $\eta^i$. On the other hand, the $\thetasu{i}$ vanish on $\Sigma_t$, so the value of $h_{ij}$ at each point of $\Sigma_t$ agrees with $\tilde h_{ij}$. In particular, $H$ restricts to the mean curvature function on $\Sigma_t$. Thus, for any parameterization $u$ of $\Sigma$,
  \[ \pDeriv{u}{t}(x) \cdot \hat n = \eta^0\left(\pDeriv{u}{t}(x)\right) = H dt\left(\pDeriv{u}{t}(x)\right) = H , \]
  and the manifolds $\Sigma_t$ satisfy mean curvature flow.

  Conversely, given a solution $N_t$ to the Mean Curvature Flow, and a fixed orthonormal coframing of $N$, the graph of $N_t$ in $\E^{n+1}$ has a unique lift to an integral manifold of $(M, \I)$.

  The Coframing \eqref{eq:MCF coframe} has been  chosen so that
  \begin{align*}
    \d \thetasu{\varnothing} & \equiv -\thetasu{a}\w\omegu^a \mod{\thetasu{\varnothing}} \\
    \d \thetasu{a} & \equiv -\pisu{ab}\w\omegu^a \mod{\thetasu{\varnothing},\thetasu{a}} \\
    \pisu{ii} & \equiv 0 \mod{\thetanotu, \thetasu{a}, \omegu^a} .
  \end{align*}
  These structure equations show that the Cartan system of $\I$ is given by
  \[ \mathcal{K} = \{ \thetasu{\varnothing}, \thetasu{a}, \omegu^a, \pisu{ab} \} . \]
  It follows from the theory of Cauchy characteristics (see, for example, \cite{BCGGG} Chapter 2.2) that:
  \begin{enumerate}
    \item The ideal $\mathcal{K}$ is Frobenius.

    \item In any small neighborhood of $M$ for which the leaf space is a manifold $M'$, there is an ideal $\I'$ on $M'$ that pulls back to $\I$. In other words, the quotient map $q\colon (M,\I) \to (M',\I')$ is an EDS morphism (so pushes solutions down to solutions).

    \item For an integral manifold $\Sigma \in M'$, the inverse image $q^{-1}(\Sigma)$ is an integral manifold of $M$. This defines a bijection between maximal integral manifolds in $M$ and $M'$.
  \end{enumerate}

  The kernel of $\mathcal{K}$ at any point is the space of Cauchy characteristic directions, and any maximal integral manifold will contain them in its tangent sub-bundle. In this case, these Cauchy directions correspond to the freedom to choose an arbitrary orthonormal coframing of each $\Sigma_t$. The map $q$ quotients down these irrelevant directions.

  Finally, it is clear from the Coframing \eqref{eq:MCF coframe} that $M'$ is a parabolic system.
\end{example}

\subsection{Independence Conditions for Parabolic Systems}

As demonstrated by the examples, it is often useful to restrict attention to integral manifolds satisfying a given transversality condition. This is clearest for integral manifolds described in example \ref{ex:jet formulation of parabolic}, where integral manifolds for which $\d x^0 \w \ldots \w \d x^n \neq 0$ are in bijection with \emph{classical solutions} of the given PDE. There are integral manifolds that don't satisfy this condition, corresponding to generalized solutions. These are also of interest, but it is clearly important to distinguish the two classes.

The condition that the form $\d x^1 \w \ldots \w \d x^n$ restrict to be non-zero is appropriately generalized in the following definition. For clarity, consider an EDS $(M,\I)$ that satisfies the constant rank condition that the degree 1 grade of $\I$,
\[ \I^1 = \I \cap \Omega^1(M) , \]
has constant rank $s$. This assumption typically holds in examples of interest, including parabolic systems.
\begin{definition}
  An \emph{independence condition} for $(M,\I)$ is a locally free $\cinfinity{M}$-submodule $J$ in $\Omega^1(M)$ of rank $n+s$ so that \textit{(i)}\enskip $\I^1 \subset J$ and \textit{(ii)}\enskip $J$ has everwhere a local basis
  \[ \thetasu{1}, \ldots, \thetasu{s}, \omegu^1, \ldots, \omegu^n \]
  for which
  \[ \I^1 = \{\thetasu{1}, \ldots, \thetasu{s}\} \quad \mbox{ and } \quad \omegu^1\w\ldots\w\omegu^n \not\in \I . \]

  An $n$-dimensional integral manifold $\Sigma$ of $M$ is a \emph{solution manifold} if $J|_\Sigma$ has rank $n$, or equivalently, if
  \[ \omegu^1\w\ldots\w\omegu^n|_\Sigma \neq 0 \]
  everywhere.
\end{definition}

Parabolic systems have a natural choice of independence condition, given by the Cauchy system of $\thetasu{\varnothing}$. It follows immediately from the structure equations that the Cauchy system of $\thetasu{\varnothing}$ is given by the ideal
\[ \cartTheta =  \{\thetasu{\varnothing},\thetasu{a},\omegu^a\}_{alg} . \]
By the general theory of Cauchy systems, the ideal $\cartTheta$ is Frobenius.
Furthermore, as $\thetanotu$ is defined up to scaling,
the ideal $\cartTheta$ is invariantly defined, independent of a choice of parabolic coframing. It plays an important role in the following.

\section{The Equivalence Problem for Parabolic Systems}\label{sec:Equivalence setup}

\subsection{\texorpdfstring{$G_0$}{G0}-structures}\label{sec:G0 structures}
When defining examples of parabolic systems, it is useful to retain the flexibility of Condition 3 in Definition \ref{def:weakly parabolic system}.
However, the space of parabolic coframes does not immediately define a $G$-structure, which will be necessary to approach the equivalence problem. So, once a parabolic system is in hand, it will be useful to strengthen Condition 3 as in the following Lemma.
\begin{lemma}\label{thm: (lem) normalized parabolic symbol}
    For any function $\lambda$ on $M$, any parabolic coframing of a parabolic system may be normalized to a new parabolic coframing for which
    \[ \pisu{ii} = \lambda\thetasu{0} .  \]
\end{lemma}
\begin{proof}
    By assumption,
    \[ \pisu{ii} \equiv 0 \mod{\thetasu{\varnothing},\thetasu{a},\omegu^a} , \]
    so that there are functions $f^{\varnothing}, f^a, f_a$ for which
    \[ \pisu{ii} = f^{\varnothing}\thetasu{\varnothing} + f^a\thetasu{a} + f_a \omegu^a . \]

    Adding any multiples of $\thetasu{\varnothing}$ and $\thetasu{a}$ to $\pisu{ij}$ again gives a parabolic coframing, so without loss of generality
    \[ \pisu{ii} = \lambda\thetasu{0} + f_a \omegu^a . \]

    If $h_{abc}$ is symmetric in its indices, then the coframe modification
    \[ \pisu{ab} \mapsto \pisu{ab} + h_{abc}\omegu^c \]
    is again parabolic. Any choice of $h_{abc}$ with $h_{iia} = -f_a$ will absorb the $\omega^a$ terms, so that $\pisu{ii} = \lambda\thetasu{0}$.
\end{proof}

\begin{definition}
    A parabolic coframing of $(M,\I)$ is \emph{$0$-adapted} if it is normalized so that $\pisu{ii} = 0$.
\end{definition}
The choice of $\lambda = 0$ is the correct normalization for generic parabolic systems, essentially because the subprincipal symbol is not well defined.
Section \ref{sec:Goursat invts} will deal with the parabolic systems for which it makes sense to ask if the subprincipal symbol vanishes. There it will be better to choose $\lambda$ to equal 1 if the subprincipal symbol is non-vanishing, or 0 if the subprincipal vanishes identically.


Fix a vector space $W = \R\oplus\R^n$ with basis $e_a = e_0, e_1 , \ldots , e_n$ and
\[ \tilde V = \R\oplus W \oplus W^\vee \oplus \Sym^2 W \cong W^\vee \oplus \bigoplus_{r=0}^2 \Sym^r W \]
with induced basis
\begin{equation}\label{eq: basis for V}
    e_{\varnothing}, e_a, e^a, e_{ab} = e_{ba} .
\end{equation}
Let $V$ be the orthogonal complement of the element $e_{ii}$, so that $\tilde V$ decomposes as
\[ \tilde V = V \oplus \R\{e_{ii}\} . \]
It is clear that a 0-adapted coframing on $M$ may be considered as a 1-form on $M$ with values in $\tilde V$.
Furthermore, at each point $x$ of $M$, the image of $T_xM$ is $V$.

One may check that any two 0-adapted coframes at a point of $M$ differ by the action of a matrix $g \in \GL(\tilde V)$ that in the basis \ref{eq: basis for V} is of the form
\begin{equation}\label{eq:G_0 structure matrix}
  \left(\begin{array}{cccc}
  k_\varnothing  & 0  & 0          & 0 \\
  (k_a) & B  & 0          & 0 \\
  (k^a) & \trans{B}^{-1}S & k_\varnothing \ \trans{B}^{-1} & 0 \\
  (k_{ab}) & D  & BT      & C_B/k_\varnothing  \end{array}\right)
\end{equation}
where
\begin{gather*}
    k_\varnothing \in \R^\times, \quad (k_a) \in W, \quad (k^a) \in W^\vee, \\
    (k_{ab}) \in \Sym^2 W \quad \mbox{ and } \quad k_{ii} = 0 \\
    B = \left( \begin{array}{cc} B_0^0 & B^i_0 \\ 0 & B^i_j \end{array} \right) \in \left( \begin{array}{cc} \R^\times & \trans{(\R^n)} \\ 0 & CO(n) \end{array} \right) \subset \GL(W)
\end{gather*}
with
\[ B^i_j = b \mathring{B}^i_j \quad \mbox{ for } \quad b \in \R^\times, \quad \mathring{B} \in SO(n) \]
and
\begin{gather*}
    S = \left( \begin{array}{cc} S^{00} & S^{0j} \\ S^{0j} & S^{ij} \end{array} \right) \in \Sym^2 W^\vee \subset \Hom(W,W^\vee) , \\
    D = (D^i_{jk}) \in \Hom(W, \Sym^2 W) \quad \mbox{ with } \quad D^a_{ii} = 0 , \\
    T = (T_{ijk}) \in \Sym^3 W \quad \mbox{ with } \quad T_{iia} = 0 ,
\end{gather*}
and finally, the matrix $C_B$ is induced by the conjugate transpose action of $B$ on $\Sym^2 W$, so that
\[ C_{B}(\pisu{ab}) = B_a^c\pisu{cd}B^d_b . \]

Let $G_0$ be the subgroup of matrices in $\GL(\tilde V)$ of this form. The canonical representation of $G_0$ on $\tilde V$ contains $V$ as a subresentation, so we may consider $G_0$ as a subgroup of $\GL(V)$.

Consider the $V$-valued coframe bundle $\coframeBundle(M)$ over $M$, whose fiber at each $x \in M$ consists of all of the linear isomorphisms
\[ u \colon T_xM \xrightarrow{\qquad} V . \]
Arbitrary local coframings of $M$ are in bijection with local sections of $\coframeBundle(M)$.
In particular, each 0-adapted coframing of $M$ gives a local section of $\coframeBundle(M)$. The set of coframes in $\coframeBundle(M)$ that come from 0-adapted coframings defines a $G_0$-structure $\gBun_0$ on $M$.

The block triangular structure of $G_0$ corresponds to the following filtration of ideals, which are adapted to the geometry of the parabolic system $(M,\I)$ and are well defined independent of any particular choice of coframing:
\begin{equation}\begin{aligned}\label{eq:coframing filtration}
\{\thetasu{\varnothing}\} & \subset \{\thetasu{\varnothing},\thetasu{i}\} \subset \{\thetasu{\varnothing},\thetasu{a}\} \subset \{\thetasu{\varnothing},\thetasu{a},\omegu^0\} \subset \{\thetasu{\varnothing},\thetasu{a},\omegu^a\} \\
& \subset \{\thetasu{\varnothing},\thetasu{a},\omegu^a,\pisu{ij}\} \subset \{\thetasu{\varnothing},\thetasu{a},\omegu^a,\pisu{aj}\} \subset \{\thetasu{\varnothing},\thetasu{a},\omegu^a,\pisu{ab}\} = \Omega^*(M) .
\end{aligned}
\end{equation}

Now recall the tautological 1-form of a $G$-structure, which can be used to calculate properties of adapted coframings in a uniform way.
\begin{definition}
On the coframe bundle $\pi\colon \coframeBundle(M) \to M$, the \emph{tautological form} $\eta \in \Omega^1(\coframeBundle(M),V)$ is defined by
\[ \eta_u(X) =  (\pi^*u)(X) \]
for all $u \in \coframeBundle(M)$ and $X \in T_u \coframeBundle(M)$.

The tautological form of a $G$-structure $\gBun \subset \coframeBundle(M)$ is the pullback of $\eta$ to $\gBun$.
\end{definition}
The tautological form is uniquely characterized by its \emph{reproducing property}, the property that
\[ s^* \eta = s \]
for any section $s$ of $\coframeBundle(M)$. For this reason, the tautological form may be thought of as a `universal' choice of coframing for $M$.
The ideal of semi-basic forms of $\gBun$, denoted $\sbforms^*$, is generated as a $\mathcal{C}^\infty(\gBun)$-module by the component 1-forms of $\eta$.

It will be useful to employ the vector notation
\[  \thetas{\varnothing},\quad \Theta=\left( \begin{array}{c} \thetas{0} \\ \thetas{i} \end{array} \right),\quad
\Omega = \left( \begin{array}{l} \omega^0 \\ \omega^i \end{array} \right),\quad
\Pi = \left( \begin{array}{cc} \pis{00} & \pis{i0} \\ \pis{i0} & \pis{ij}  \end{array} \right) \]
for the components of the tautological form on $\gBun_0$.
With analogous notation for a $0$-adapted coframing, the structure equations can be written more concisely as
\begin{equation}\begin{aligned} \label{eq:basic structure eqs}
\d\thetasu{\varnothing} & \equiv -\trans{\underline{\Theta}}\w\underline{\Omega} \mod{\thetasu{\varnothing}} \\
\d\underline{\Theta} & \equiv -\underline{\Pi}\w\underline{\Omega} \mod{\thetasu{\varnothing},\thetas{a}} .
\end{aligned}\end{equation}

\emph{Cartan's first structure equation} states that on a given $G_0$-structure $\gBun_0$, there is a pseudo-connection
\[ \varphi \in \Omega^1(\gBun_0, \mathfrak{g}_0) \]
(for $\mathfrak{g}_0$ the Lie algebra of $G_0$) and a torsion map
\[ T \colon \gBun_0 \xrightarrow{\quad} \Hom(\Lambda^2 V, V) \]
so that
\[ \d \eta = - \varphi \w \eta + T(\eta\w\eta) . \]
Roughly, $\varphi$ measures the variation of $\eta$ in the fiber direction and $T$ measures the first order twisting between fibers. In vector notation,
\begin{equation}\label{eq:0 adapted Cartan structure} \d \left(\begin{array}{l}\thetas{\varnothing} \\ \Theta \\ \Omega \\ \Pi \end{array}\right) =
-\left( \begin{array}{cccc}
\kappa_\varnothing  & 0 & 0 & 0 \\
(\kappa_a) & \beta & 0 & 0 \\
(\kappa^a) & \sigma & \kappa_\varnothing \Id_\Omega  - \trans{\beta} & 0 \\
(\kappa_{ab}) & \gamma & \tau & C_\beta - \kappa_\varnothing \Id_\Pi  \end{array}\right)\wedge\
\left(\begin{array}{l}\thetas{\varnothing} \\ \Theta \\ \Omega \\ \Pi \end{array}\right)
+ \left(\begin{array}{l}T_{\thetas{\varnothing}} \\ T_{\Theta} \\ T_{\Omega} \\ T_{\Pi} \end{array}\right) ,
\end{equation}
where the torsion terms $T_{\thetas{\varnothing}},T_\Theta,T_\Omega$ and $T_\Pi$ are semi-basic.
The matrix valued one-forms in Equation \eqref{eq:0 adapted Cartan structure} have components as follows:
\[ \kappa_\varnothing \in \Omega^1(\gBun_0, \R), \quad (\kappa_a) \in \Omega^1(\gBun_0, W), \quad (\kappa^a) \in \Omega^1(\gBun_0, W^\vee), \]
\[ (\kappa_{ab}) \in \Omega^1(\gBun_0, \Sym^2 W) \quad \mbox{ and } \quad \kappa_{ii} = 0 \]

\[ \beta = \left( \begin{array}{cc} \beta_0^0 & \beta_0^i \\  0  & \beta^i_j \end{array} \right) \in
\Omega^1\left(\gBun_0 , \left( \begin{array}{cc} \R & \trans{(\R^n)} \\ 0 & \mathfrak{co}(n) \end{array} \right)\right) \]
where $\beta^i_j$ decomposes into an $\mathfrak{so}(n)$-valued component $\mathring{\beta}$ and conformal trace component $\beta_{\tr}$,
\[ \beta^i_j = \mathring{\beta}^i_j + \delta^i_j \beta_{\tr} . \]
Denote
\[ \sigma = \left( \begin{array}{cc} \sigma^{00} & \sigma^{0j} \\ \sigma^{0j} & \sigma^{ij} \end{array} \right) \in \Omega^1\left( \gBun_0, \Sym^2 W^\vee \right). \]
Denote the components of $\gamma$ by
\[ \gamma = (\gamma^a_{bc}) \in \Omega^1(\gBun_0, \Hom(W, \Sym^2 W)) \quad \mbox{ with } \quad \gamma^a_{ii} = 0 . \]
Finally, the action of $C_\beta$ is given by left and right multiplication by $\beta$:
\[ C_\beta(\pi_{ab}) = \beta_a^c\w\pi_{cb} - \pi_{ac}\w\beta^c_b . \]

The reproducing property of the tautological form immediately determines some of the torsion forms. Because \eqref{eq:basic structure eqs} holds for any $0$-adapted coframing $\etu$,
\[ \etu^*(T_{\thetas{\varnothing}}) \equiv \etu^*\d\thetas{\varnothing} = \d\thetasu{\varnothing} \equiv -\trans{\underline{\Theta}}\w\underline{\Omega}  \mod{\thetasu{\varnothing}} , \]
and thus
\[ T_{\thetas{\varnothing}} = -\trans{\Theta}\w\Omega  + \xi_\varnothing\w\thetas{\varnothing} \]
for a semi-basic 1-form $\xi_\varnothing$. Adding $\xi_\varnothing$ to $\kappa_\varnothing$ will not affect Cartan's structure equation, but will absorb the torsion. Doing so simplifies the first component of \eqref{eq:0 adapted Cartan structure} to
\[ \d\thetas{\varnothing} = -\kappa_\varnothing \w\thetas{\varnothing} - \trans{\Theta}\w\Omega . \]
It is clear that no other modification of $\kappa_\varnothing $ can be made to absorb the remaining torsion, and that $\kappa_\varnothing $ is uniquely defined up to a multiple of $\thetas{\varnothing}$.

An analogous calculation shows that there is a matrix of semi-basic 1-forms
\[ \xi=\left( \begin{array}{cc} \xi^0_0 & \xi^j_0 \\ \xi_i^0 & \xi^j_i \end{array} \right) \]
so that
\[ \d{\Theta} = -(\kappa_a)\w\thetas{\varnothing} - \beta\w\Theta -\Pi\w\Omega - \xi\w\Theta . \]
From this equation it is clear that semi-basic forms may be added to $\beta$ in such a way that the torsion $\xi$ reduces to
\[ \xi=\left( \begin{array}{cc} 0 & 0 \\ \xi_i^0 & \xi^j_i \end{array} \right) , \]
where furthermore $\xi_i^j$ is symmetric and trace-free,
\[ \xi^i_j = \xi^j_i \quad \mbox{ and } \quad \xi_i^i = 0 . \]

The torsion forms $\xi$ also control the torsions of the forms $\omega^a$. Indeed, from the structure equations,
\begin{align*} 0 = \d^2\thetas{\varnothing} & \equiv d(-\kappa_\varnothing \w\thetas{\varnothing} - \trans{\Theta}\w\Omega) \\
& \equiv \trans{\Theta}\w(-\trans{\xi}\w\Omega + T_\Omega)
\mod{\thetas{\varnothing}} ,
\end{align*}
where in the last line, the sign on the first term arises because $\trans{(\xi \w \Theta)} = -\trans{\Theta}\w\trans{\xi}$.
An application of the generalized Cartan's lemma shows that
\[ T_\Omega \equiv H\w\Theta + \trans{\xi}\w\Omega \mod{\thetas{\varnothing}} , \]
where $H$ is a semi-basic, $\Sym^2 W^\vee$-valued 1-form. Thus
\[ \d\Omega \equiv - (\sigma-H)\w\Theta - (\kappa_\varnothing  - \trans{\beta})\w\Omega + \trans{\xi}\w\Omega \mod{\thetas{\varnothing}} . \]
By modifying $\sigma$ accordingly, all of the torsion $H$ may be absorbed.

To summarize the structure equations so far,
\begin{align*}
\d\thetas{\varnothing} & = -\kappa_{\varnothing}\w\thetas{\varnothing} - \thetas{a}\w\omega^a , \\
\d\thetas{i} & = -\kappa_i\w\thetas{\varnothing} -\beta_i^j\w\thetas{j} - \xi_i^a\w\thetas{a} - \pis{ia}\w\omega^a , \\
\d\thetas{0} & \equiv - \beta_0^a\w\thetas{a} - \pis{0a}\w\omega^a && \mod{\thetas{\varnothing}} , \\
\d\omega^0 & \equiv - \sigma^{0a}\w\thetas{a} - (\kappa_{\varnothing} - \beta_0^0)\w\omega^0 + {\xi}_i^0\w\omega^i && \mod{\thetas{\varnothing}} \\
\d\omega^i & \equiv - \sigma^{ia}\w\thetas{a} - (\delta^i_j\kappa_{\varnothing} - \beta^i_j)\w\omega^j + \beta_0^i\w\omega^0 + \xi^i_j\w\omega^j && \mod{\thetas{\varnothing}} \\
\d\pis{ab} & \equiv - \gamma_{ab}^c\w\thetas{c} - \tau_{abc}\w\omega^c + \kappa_{\varnothing}\w\pis{ab} - \beta_a^c\w\pis{cb} + \pis{ac}\w\beta_b^c  + T_{\pis{ab}} && \mod{\thetas{\varnothing}}
\end{align*}

Finally, $\xi$ also determines the `highest weight' torsion of $\Pi$. For example, one computes that
\[ 0 = \d^2\thetas{i} \equiv (-T_{\pis{ij}} - \xi_i^a\w\pis{aj} + \pis{ia}\w\xi^a_j)\w\omega^j \mod{\thetas{\varnothing}, \thetas{a}, \omega^0} \]
and thus by Cartan's Lemma,
\[ T_{\pis{ij}} \equiv -\xi_i^a\w\pis{aj} + \pis{ia}\w\xi^a_j \mod{\cartTheta} . \]

\section{The Monge-Amp\`ere Invariants for Parabolic Systems}\label{sec:MA invariants}
The first invariants one finds for parabolic systems derive from the 1-forms $\xi_i^a$. Because the $\xi_i^a$ are semi-basic, there are functions $V_i^{abc}$ on $\gBun_{0}$ so that
\[ \xi_i^a \equiv V_i^{abc}\pis{bc} \mod{\cartTheta} . \]
As both $\xi_i^j$ and $\pis{kl}$ are symmetric and traceless, the functions $V_i^{jkl}$ are symmetric and (spatially) traceless in the $ij$ indices as well as the $kl$ indices. Likewise, $V_i^{0kl}$ is symmetric and spatially traceless in $kl$ indices.
Further symmetries follow from
\[ \pis{ii} \equiv 0 \mod{\cartTheta} , \]
for then
\[ 0 \equiv \d\pis{ii} \equiv -2\xi_i^a\w\pis{ai} \equiv -2V_i^{abc}\pis{bc}\w\pis{ai} \mod{\cartTheta} \]
and thus
\[ V_i^{abc} = V_{b}^{cia} , \]
with the understanding that $V_i^{a00} = 0$ identically for each $i$ and $a$.

\subsection{Primary Monge-Amp\`ere Invariants}
The primary Monge-Amp\`ere invariants are the non-absorbable components of $V^{0j0}_i$ and $V^{0jk}_i$, the coefficients of $\xi_i^0$ in
\[ \xi_i^0 \equiv V^{0j0}_i\pis{j0} + V^{0jk}_i\pis{jk} \mod{\thetas{\varnothing},\thetas{a},\omega^a} . \]

To see how these functions vary in each fiber, consider the exterior derivative of Cartan's structure equation. In indices,
\[ \d\thetas{i} \equiv -\beta_i^j\w\thetas{j} - \xi^j_i\w\thetas{j} - \xi_i^0\w\thetas{0} - \pis{ij}\w\omega^j - \pis{i0}\w\omega^0 \mod{\thetas{\varnothing}} . \]
Taking the exterior derivative,
\[ 0 \equiv -\left(\beta_i^j\w\xi_j^0 + d(\xi_i^0) + \xi_i^0\w\beta_0^0 + \pis{ij}\w\sigma^{j0} + \pis{i0}\w\sigma^{00} \right)\w\thetas{0}
\mod{\thetas{\varnothing},\thetas{i},\omega^a,\sbforms^3} , \]
so that by Cartan's Lemma,
\begin{equation}\label{eq:MA variation} 0 \equiv d(\xi_i^0) + \beta_i^j\w\xi_j^0 + \xi_i^0\w\beta_0^0 + \pis{ij}\w\sigma^{j0} + \pis{i0}\w\sigma^{00}
  \mod{\cartTheta,\sbforms^2} .
\end{equation}

At the `highest weight', after plugging in for $\xi_i^0$, equation \eqref{eq:MA variation} simplifies to
\[ 0 \equiv (d(V_i^{0j0}) + \beta_i^kV^{0j0}_k - V_i^{0k0}\beta_k^j + V_i^{0j0}(\kappa_\varnothing -2\beta^0_0) - \delta_i^j\sigma^{00})\w\pis{j0} \mod{\cartTheta,\pis{ij},\sbforms^2} , \]
which implies that
\[ d(V_i^{0j0}) \equiv -\beta_i^k V^{0j0}_k + V_i^{0k0}\beta_k^j + (2\beta^0_0 - \kappa_\varnothing)V_i^{0j0} + \delta_i^j\sigma^{00} \mod{\sbforms^1} \]
for all $i$ and $j$.
Integrating this, one finds that the function $(V_i^{0j0})$ on $\gBun_0$ is $G_0$-equivariant\footnote{To be precise, this argument only works for the identity component of $G_0$. However, one can check the variation of $(V_i^{0j0})$ for one element in each component of $G_0$ to see that it really is a relative invariant.}. Indeed, for a 0-coframe $u$ and a matrix $g$ as in \eqref{eq:G_0 structure matrix},
\[ \left(V_i^{0j0}(g\cdot u)\right) = \frac{(B^0_0)^2}{k_\varnothing}(b\mathring{B})^{-1}\left(V_i^{0j0}(u)\right)b\mathring{B} + S^{00}\delta_i^j = \frac{(B^0_0)^2}{k_\varnothing}\mathring{B}_i^k\left(V_k^{0l0}(u)\right)\mathring{B}^j_l + S^{00}\delta_i^j . \]
Due to the last term, there are choices of coframe for which $V_i^{0j0}$ is traceless in $ij$. The subbundle of such coframes has structure group consisting of matrices as in \eqref{eq:G_0 structure matrix} such that $S^{00} = 0$. When restricted to this reduced $G$-structure, the traceless component of $(V_i^{0j0})$ becomes a relative invariant. In particular, if $(V_i^{0j0})$ vanishes at a point in one coframing, then it vanishes in all coframings at that point. This is by definition the \emph{highest weight primary Monge-Amp\`ere invariant}. Note that the pseudo-connection form $\sigma^{00}$ becomes semi-basic in the reduction.

Now suppose that $(V_i^{0j0})$ vanishes identically and the coframe reduction has been carried out. Then \eqref{eq:MA variation} simplifies to
\begin{align*}
     0 \equiv \left(\d V^{0jk}_i + (\kappa_\varnothing  - \beta^0_0)V^{0jk}_i + \beta_i^l V_l^{0jk} - V_i^{0lk} \beta_l^j \right. & \\
     \left. - V_i^{0jl} \beta_l^k - \frac{1}{2}\delta_i^j\sigma^{0k} - \frac{1}{2}\delta_i^k\sigma^{0j} + \frac{1}{n}\delta^{jk}\sigma^{0i} \right) & \w\pis{jk} \mod{\cartTheta,\sbforms^2} ,
\end{align*}
so that
\begin{align*}
    \d V^{0jk}_i \equiv -\beta_i^l V_l^{0jk} & + V_i^{0lk} \beta_l^j + V_i^{0jl} \beta_l^k + (\beta^0_0 - \kappa_\varnothing)V^{0jk}_i \\
    & + \frac{1}{2}\delta_i^j\sigma^{0k} + \frac{1}{2}\delta_i^k\sigma^{0j} - \frac{1}{n}\delta^{jk}\sigma^{0i} \mod{\sbforms^1} .
\end{align*}
Integrating, $(V^{0jk}_i)$ transforms as
\[ \left(V^{0jk}_i(g\cdot u)\right) = \frac{b B_0^0}{k_\varnothing} \mathring{B}_i^l\left(V^{0mn}_l(u)\right)\mathring{B}_m^j\mathring{B}_n^k + \frac{1}{2}\delta_i^j S^{0k} + \frac{1}{2}\delta_i^k S^{0j} - \frac{1}{n}\delta^{jk}S^{0i} \]
in each fiber.
A coframe adaptation may be made to absorb the trace components of this representation, so that
\[ V_i^{0ik} = 0 \]
for each $k$.
Such coframes are called \emph{1-adapted}.
After this coframe adaptation, the remaining components of $(V^{0jk}_i)$ are relative invariants, the lowest weight piece of the primary Monge-Amp\`ere invariant.

The subbundle of 1-adapted coframes $\gBun_1 \subset \gBun_0$ has structure group $G_{1}$, which consists of matrices as in \eqref{eq:G_0 structure matrix} so that
\[ S^{00} = S^{0i} = 0 . \]
Upon restricting to $\gBun_1$, the pseudo-connection forms $\sigma^{0a}$ become semi-basic, contributing new non-absorbable torsion.

It has just been shown that all of the primary Monge-Amp\`ere invariants vanish if and only if there are parabolic coframings so that
\[ \xi_i^0 \equiv 0 \mod{\cartTheta} , \]
in which case
\[ 0 = \d^2\thetas{i} \equiv (\sigma^{0a}\w\pis{ia})\w\thetas{0} \mod{\thetas{\varnothing},\thetas{i},\omega^a} . \]
Then
\[ 0 \equiv \sigma^{0a}\w\pis{ia} \mod{\cartTheta} , \]
and, assuming $n\geq 3$, one may apply Cartan's Lemma three times, with $i = 1, 2, 3$ to deduce that
\[ \sigma^{0a} \in \{\cartTheta, \pis{1b}\} \cap \{\cartTheta, \pis{2b}\} \cap \{\cartTheta, \pis{3b}\} = \cartTheta \]
so that
\[ \sigma^{0a}\equiv 0 \mod{\cartTheta} . \]

\subsection{Secondary Monge-Amp\`ere Invariants}
Consider a parabolic system with vanishing primary Monge-Amp\`ere invariants and assume the $G_1$-reduction of coframing has been made, so that $V^{0j0}_i = V^{0jk}_i = 0$. There are $V_i^{jkl}$ so that
\[ \xi_i^j \equiv V_i^{jkl}\pis{kl} \mod{\cartTheta} , \]
and the non-absorbable components of these functions are the \emph{secondary Monge-Amp\`ere invariants}. Since the primary Monge-Amp\`ere invariants vanish,
\[ \d\thetas{i} \equiv - \beta_i^j\w\thetas{j} - \xi^j_i\w\thetas{j} - \pis{ij}\w\omega^j - \pis{i0}\w\omega^0 \mod{\thetas{\varnothing},\Lambda^2 \cartTheta} , \]
and thus
\begin{multline*}
    0 \equiv -\left(\d(\xi_i^j) + \xi_i^k\w\beta_k^j + \beta_i^k\w\xi_k^j + \pis{ik}\w\sigma^{kj} \right. \\
    \left. + \d\beta_i^j + \beta_i^k\w\beta_k^j \right)\wedge\thetas{j} \mod{\thetas{\varnothing},\omega^a,\Lambda^2 \cartTheta,\sbforms^3} .
\end{multline*}
By an application of Cartan's lemma,
\begin{equation}\label{eq: total xi_i^j variation}
  0 \equiv \d(\xi_i^j) + \xi_i^k\w\xi_k^j + \xi_i^k\w\beta_k^j + \beta_i^k\w\xi_k^j + \pis{ik}\w\sigma^{kj} + \d\beta_i^j + \beta_i^k\w\beta_k^j \mod{\cartTheta,\sbforms^2}
\end{equation}
One would have to prolong the $G_1$-structure to properly describe the $\d\beta + \beta\w\beta$ term, but that won't be necessary here. Instead, consider just the component of \eqref{eq: total xi_i^j variation} that is symmetric and traceless in $i$ and $j$,
\begin{equation}\label{eq: xi_i^j variation}
    0 \equiv \d(\xi_i^j) + \xi_i^k\w\beta_k^j + \beta_i^k\w\xi_k^j + \frac{1}{2}(\pis{ik}\w\sigma^{kj} + \pis{jk}\w\sigma^{ki})_{0} \mod{\cartTheta,\sbforms^2} .
\end{equation}
This equation determines the variation of the secondary Monge-Amp\`ere invariants in each fiber.

Substituting for $\xi_i^j$, Equation \ref{eq: xi_i^j variation} simplifies to
\begin{align*}
    0 \equiv & \left( \d V_i^{jkl} + \kappa_\varnothing V_i^{jkl} + \beta_i^m V_m^{jkl} - V_i^{mkl} \beta_m^j - V_i^{jml} \beta_m^k  - V_i^{jkm} \beta_m^l \right. \mod{\cartTheta, \sbforms^2} \\
    & \left. +\frac{1}{4}(\delta_i^k\sigma^{jl}+\delta_i^l\sigma^{jk}+\delta_j^k\sigma^{il}+\delta_j^l\sigma^{ik}) - \frac{1}{n}\delta_i^j\sigma^{kl} - \frac{1}{n}\delta_k^l\sigma^{ij} + \frac{1}{n^2}(\delta_i^j\delta_k^l)\sigma^{mm} \right) \w \pis{kl}  ,
\end{align*}
so that
\begin{align*}
    \d V_i^{jkl} \equiv & -\kappa_\varnothing V_i^{jkl} - \beta_i^m V_m^{jkl} + V_i^{mkl} \beta_m^j + V_i^{jml} \beta_m^l + V_i^{jkm} \beta_m^l \mod{\sbforms^1} \\
    & - \frac{1}{4}(\delta_i^k\sigma^{jl}+\delta_i^l\sigma^{jk}+\delta_j^k\sigma^{il}+\delta_j^l\sigma^{ik}) + \frac{1}{n}\delta_i^j\sigma^{kl} + \frac{1}{n}\delta_k^l\sigma^{ij} - \frac{1}{n^2}(\delta_i^j\delta_k^l)\sigma^{mm}  .
\end{align*}
Integrating, $V_i^{jkl}$ varies as
\begin{multline*}
    \left(V_i^{jkl}(g\cdot u)\right) = \frac{b^2}{k_\varnothing} \mathring{B}_i^m \left(V_m^{npq}(u)\right) \mathring{B}_n^j\mathring{B}_p^k\mathring{B}_q^l - \frac{1}{4}(\delta_i^kS^{jl}+\delta_i^lS^{jk}+\delta_j^kS^{il}+\delta_j^lS^{ik}) \\ + \frac{1}{n}\delta_i^j S^{kl} + \frac{1}{n}\delta_k^l S^{ij} - \frac{1}{n^2}(\delta_i^j\delta_k^l)S^{mm}.
\end{multline*}
As such, there is a reduction of coframes so that
\[ V_i^{jil} = 0 \]
for all $j$ and $l$, which reduces $\gBun_1$ to a $G_{MA}$-structure $\gBun_{MA}$, where $G_{MA}$ is the subgroup of matrices as in \eqref{eq:G_0 structure matrix} for which
\[ S = 0 . \]
The pseudo-connection forms $\sigma^{ij}$ are semi-basic when restricted to $\gBun_{MA}$. The remaining un-absorbed components of $V_i^{jkl}$ are the secondary Monge-Amp\`ere invariants, which take values in $\Sym_0^2 (\Sym_0^2 \R^n)$.

To summarize, we have nearly proved the following Proposition.
\begin{proposition}\label{thm: (prop) MA invariants vanish if xi, sigma vanish mod}
    The primary and secondary Monge-Amp\`ere invariants of a parabolic system $M$ vanish identically if and only if there is a reduction of coframings to a $G_{MA}$-structure $\gBun_{MA}$ so that the torsion form $\xi$ satisfies
    \[ \xi_i^0 \equiv \xi_i^j \equiv 0 \mod{\cartTheta} . \]

    Furthermore, after reducing coframes, the torsion $\sigma$ satisfies
    \[ \sigma^{00} \equiv \sigma^{i0} \equiv 0 \mod{\cartTheta} , \]
    \[ \sigma^{ij} \equiv V\pis{ij} \mod{\cartTheta} , \]
    for a single function $V$, the \emph{tertiary Monge-Amp\`ere invariant}.
\end{proposition}
\begin{proof}
    Only the final statement remains to be proven. After reduction to $\gBun_{MA}$, we have
    \[ \d\thetas{i} \equiv  - \beta_i^j\w\thetas{j} - \pis{ij}\w\omega^j - \pis{i0}\w\omega^0 \mod{\thetas{\varnothing},\Lambda^2 \cartTheta} , \]
    so from $\d^2 \theta_i = 0$,
    \[ 0 \equiv -\left(\pis{ik}\w\sigma^{kj} + \d\beta_i^j + \beta_i^k\w\beta_k^j \right)\wedge\thetas{j} \mod{\thetas{\varnothing},\omega^a,\Lambda^2 \cartTheta} . \]
    Taking the $ij$-symmetric and traceless component gives
    \begin{equation}\label{eq:sigma pi}
        0 \equiv (\pis{ik}\w\sigma^{kj} + \pis{jk}\w\sigma^{ki})_0 \mod{\cartTheta} .
    \end{equation}
    Consider first the special case $i=j$, so that (summing over $k$ but not $i$)
    \[ 0 \equiv \pis{ik}\w\sigma^{ki} \mod{\cartTheta} . \]
    Applying Cartan's Lemma, one finds that for each fixed $i, k$,
    \[ \sigma^{ik} \equiv 0 \mod{\pis{i1}, \ldots, \pis{in}, \cartTheta} . \]
    But then from the $ik$ symmetry of $\sigma^{ik}$,
    \[ \sigma^{ik} \equiv 0 \mod{\pis{k1}, \ldots, \pis{kn}, \cartTheta} , \]
    so in fact,
    \[ \sigma^{ik} \equiv  V_{ik}\pis{ik} \mod{\cartTheta} \]
    for scalars $ V_{ik}$.
    Plugging this into Equation \ref{eq:sigma pi} for distinct $i$ and $j$, it follows immediately that for each $k$,
    \[  V_{ik} =  V_{jk} , \]
    so that all are equal to single function $ V$.

    Now,
    \[ \d \omega^i \equiv - V\pis{ij}\w\thetas{j} - (\delta^i_j \kappa_{\varnothing} - \beta^i_j)\w\omega^j + \beta^i_0\w\omega^0 \mod{\thetas{\varnothing}, \Lambda^2 \cartTheta} , \]
    and taking the derivative of both sides,
    \[ 0 \equiv -(\d V -  V(2\kappa_{\varnothing} - 4\beta_{\tr}))\w\pis{ij}\w\thetas{j} \mod{\thetas{\varnothing}, \omega^a, \Lambda^2 \cartTheta, \sbforms^3} . \]
    It follows that $ V$ is a relative invariant, which vanishes for one MA-adapted coframings if and only if it vanishes in all MA-adapted coframings.
\end{proof}

\section{Monge-Amp\`ere Systems}\label{sec:Monge-Ampere Systems}
\subsection{Monge-Amp\`ere equations}
The Monge-Amp\`ere invariants just developed can be used to characterize paraboblic equations of Monge-Amp\`ere type.
Recall that a Monge-Amp\`ere equation is a \nth{2}-order equation of the form
\begin{equation}\label{eq:MA differential equation}
  F(x^a,u,p_a,p_{ab}) = \sum_{\substack{I, J \subseteq \{0,\ldots,n\} \\ |I|=|J|}} A_{I,J}(x^a,u,p_a) H_{I,J} = 0 ,
\end{equation}
where $H_{I,J}$ is the row $I$, column $J$ minor sub-determinant of the Hessian of $u$.

The following definition is from \cite{BryantGriffithsGrossman:EDSandEulerLagrangePDEs}, in the context of Euler-Lagrange equations.
\begin{definition}\label{def: Monge Ampere EDS system}
  A \emph{Monge-Amp\`ere system} in $n+1$ variables is a $2n+3$ dimensional exterior differential system $(M,\I)$ such that $\I$ is locally generated by a 1-form $\thetasu{\varnothing}$ and an $(n+1)$-form $\Upsilon$ satisfying:
  \begin{enumerate}
    \item $\thetasu{\varnothing}$ is maximally non-integrable:
    \[ \thetasu{\varnothing} \w (\d\thetasu{\varnothing})^{n+1} \neq 0 . \]
    \item At each point of $M$,
    \[ \Upsilon \not\equiv 0 \mod{\thetasu{\varnothing}, \d\thetasu{\varnothing}} . \]
  \end{enumerate}
\end{definition}
Equivalently, such systems are defined by local coframings
\[ \thetasu{\varnothing}, \omegu^a, \pisu{a} \]
so that
\[ \d\thetasu{\varnothing} \equiv - \pisu{a}\w\omegu^a \mod{\thetasu{\varnothing}} \]
and
\begin{equation}\label{eq:Upsilon in coframing}
  \Upsilon \equiv \sum_{\substack{I, J \subseteq \{0,\ldots,n\}  \\ |I|=|J|}} A_{I,J}\pisu{I} \w \omegasu{J} \mod{\thetasu{\varnothing} }
\end{equation}
for functions $A_{I,J}$ on $M$. Here
\[ \pisu{I} = \prod_{a \in I} \pisu{i} \]
and $\omegasu{J}$ is the omitted index notation,
\[ \omegasu{J} = \pm\prod_{a \not \in J} \omegu^a \]
with signs specified by the condition that
\[ \omegu_{J}\w\omegasu{J} = \omegasu{\varnothing} . \]

As an element of $\I$, the $(n+1)$-form $\Upsilon$ is only defined up to scaling and addition of multiples of $\thetasu{\varnothing}$ and $\d\thetasu{\varnothing}$. However, by the pointwise Lefschetz decomposition from symplectic linear algebra\footnote{See, for example, \cite{BryantGriffithsGrossman:EDSandEulerLagrangePDEs}, Proposition 1.1.}, there is a unique multiple $\gamma\w\d\thetasu{\varnothing}$ so that $\Upsilon + \gamma\w\d\thetasu{\varnothing}$ is \emph{primitive},
\[ (\Upsilon + \gamma\w\d\thetasu{\varnothing})\w\d\thetasu{\varnothing} \equiv 0 \mod{\thetasu{\varnothing}} . \]
So, without loss of generality, assume that
\[ \Upsilon\w\d\thetasu{\varnothing} \equiv 0 \mod{\thetasu{\varnothing}} . \]
Furthermore, $\d\Upsilon$ is an $n+2$-form, so again by Lefschetz, there is a unique (modulo $\thetanotu$) $n$-form $\gamma'$ so that
\[ \d\Upsilon \equiv \d\thetanotu\w\gamma' \mod{\thetanotu} . \]
Subtracting $\thetanotu\w\gamma'$ from $\Upsilon$, one arranges that
\[ \d\Upsilon \equiv 0 \mod{\thetanotu} . \]
With this assumption, the representative $\Upsilon$ in $\I$ is uniquely defined up to scaling. In the following it will be assumed without comment that this normalization has been carried out.

At each level $r = |I| = |J| > 0$ in Equation \ref{eq:Upsilon in coframing}, the functions $A_{I,J}$ collectively take values in $\Lambda^r \R^{n+1} \otimes \Lambda^r \R^{n+1}$. Since $\Upsilon$ is assumed primitive, the $A_{I,J}$ furthermore satisfy a generalized Bianchi identity, taking values in the kernel of
\[ \Lambda^r \R^{n+1} \otimes \Lambda^r \R^{n+1} \xrightarrow{\qquad} \Lambda^{r+1} \R^{n+1} \otimes \Lambda^{r-1} \R^{n+1} . \]
This implies in particular that $A_{I,J} = A_{J,I}$ in equation \eqref{eq:Upsilon in coframing}.

Monge-Amp\`ere systems model the solutions of Monge-Amp\`ere equations, as follows. Let $M = J^1(\R^{n+1},\R)$ and fix the coframing
\[ \thetasu{\varnothing} = \d u - p_a \d x^a ,\quad \omegu^a = \d x^a,\quad \pisu{a} = \d p_a . \]
Corresponding to equation \eqref{eq:MA differential equation}, define the $(n+1)$-form
\[ \Upsilon = \sum_{|I|=|J|} A_{I,J}(x^a,u,p_a)\, \pisu{I} \w \omegasu{J} . \]
The Monge-Amp\`ere system $(M,\{\thetasu{\varnothing},\Upsilon\})$ has a natural independence condition, defined by the forms $\omegu^a$. Then any solution manifold is locally the 1-jet graph of a function $u(x^a)$, and the condition that $\Upsilon$ vanish forces $u$ to be a solution to \eqref{eq:MA differential equation}.

Conversely, any Monge-Amp\`ere system $(M,\{\thetasu{\varnothing},\Upsilon\})$ is locally equivalent to some Monge-Amp\`ere differential equation: By the Pfaff Normal Form Theorem, there are local coordinates $x^a, u, p_a$, and a nonzero function $\lambda$ on $M$ so that
\[ \thetasu{\varnothing} = \lambda(\d u - p_a \d x^a) . \]
The $\d x^a$'s determine an independence condition, albeit not contact invariantly. Fix a local coframing
\[ \thetasu{\varnothing} = \d u - p_a \d x^a ,\quad \omegu^a = \d x^a,\quad \pisu{a} = \d p_a . \]
There are functions $A_{I,J}$ so that
\[ \Upsilon = \sum_{|I|=|J|} A_{I,J}(x^a,u,p_a)\, \pisu{I} \w \omegasu{J} . \]

Consider a solution manifold $\Sigma$ of $M$. Since
\[ \left.\thetasu{\varnothing}\right|_{\Sigma} = 0 \quad \mbox{ and } \quad \d x^0\w\ldots\w\d x^n|_\Sigma \neq 0 , \]
$\Sigma$ is locally the graph of functions $u(x^a)$ and $p_a(x^a)$ so that
\[ p_a = \pDeriv{u}{x^a} \]
and thus
\[ \d p_a = \pDeriv{^2u}{x^a \partial x^b} \d x^b . \]
The condition that $\Upsilon$ vanishes when restricted to $\Sigma$ is equivalent to the condition that $u(x^a)$ solves the equation
\[ \sum_{|I|=|J|} A_{I,J} H_{I,J} = 0 . \]

The class of Monge-Amp\`ere systems is manifestly preserved by EDS equivalences, making clear that the class of Monge-Amp\`ere equations is preserved under changes of variables.

\subsection{Quasi-Parabolic Monge-Amp\`ere systems}

\begin{definition}
    A Monge-Amp\`ere EDS $(M, \I)$ as in Definition \ref{def: Monge Ampere EDS system} is \emph{quasi-parabolic} if
    there is (locally) a 1-form $\omegu^0$ on $M$ independent of $\thetasu{\varnothing}$ and so that
    \[ \omegu^0\w\Upsilon \equiv 0 \mod{\thetasu{\varnothing}, \d\thetasu{\varnothing}} , \]
    and any other such 1-form is a linear combination of $\thetasu{\varnothing}$ and $\omegu^0$.
\end{definition}

In this case, there are local Monge-Amp\`ere coframings
\[ \thetasu{\varnothing}, \omegu^a, \pisu{a} \]
extending $\thetasu{\varnothing}, \omegu^0$. From the fact that $\omegu^0\w\Upsilon \equiv 0$ and the symmetry $A_{I,J} = A_{J,I}$ it follows that
\begin{equation}\label{eq:Upsilon in parabolic coframing}
  \Upsilon \equiv \sum_{\substack{I, J \subseteq \{1,\ldots,n\} \\ |I|=|J|}} A_{I,J}\pisu{I} \w \omegasu{J} \mod{\thetasu{\varnothing}, \d\thetasu{\varnothing}} ,
\end{equation}
which differs from Equation \ref{eq:Upsilon in coframing} in the detail that the index sets $I$ and $J$ can no longer contain $0$.

If it happens that $\d\omegu^0\w\omegu^0 = 0$, then there are local choices of coordinates $x^a, u, p_a$ such that
\[ \thetasu{\varnothing} = \lambda(\d u - p_a \d x^a) \]
and
\[ \omegu^0 = \d x^0 . \]
The induced Monge-Amp\`ere differential equation will depend only on second derivatives in spatial directions.

The modifier \emph{quasi} is necessary because the symbol of a non-linear equation can depend on which solution one linearizes at. Consider, for example, the Monge-Amp\`ere equation
\[ F = p_0 - \sum_{i, j = 1}^n (p_{ii}p_{jj} - p_{ij}^{\ 2}) = 0 . \]
From
\[ \d F = \d p_0 - \sum_{i, j = 1}^n (p_{ii}\d p_{jj} + p_{jj}\d p_{ii} - 2p_{ij}\d p_{ij}) \]
it is clear that the symbol of $F$ depends on which 2-jet of a solution on linearizes around, and is only positive semi-definite in an open set of 2-jets of solutions.

In terms of exterior differential systems, a generic quasi-parabolic Monge-Amp\`ere equation $(M_{-1}, \I_{-1})$ will have prolongation $(\tilde{M}, \I)$, where $\tilde{M}$ satisfies conditions (1) and (2) of Definition \ref{def:weakly parabolic system}, but (3) can only hold in an open subset.

As a manifold, $\tilde{M} \subset \Gr_{n+1}(T M_{-1}, \omegasu{\varnothing})$ is the subset of $(n+1)$-planes $E$ that satisfy the independence condition and the condition $\left.\I_{-1}\right|_E = 0$.
The independence condition guarantees that elements $E$ in the fiber over $y \in M_{-1}$ are parameterized by numbers $p_{ab}$ so that
\[ E = \{\thetasu{\varnothing}, \pisu{a} - p_{ab}\omegu^b\}_y^\perp . \]
This defines the canonical contact forms
\[ \thetasu{a} = \pisu{a} - p_{ab}\omegu^b \]
on $\Gr_{n+1}(T M_{-1}, \omegasu{\varnothing})$, and the ideal $\I$ is the pullback to $\tilde{M}$ of
\[ \{ \thetasu{\varnothing}, \thetasu{a} \} . \]

The condition $\d\thetasu{\varnothing}|_E = 0$ ensures that $p_{ab} = p_{ba}$, while $\Upsilon|_E = 0$ induces a single functional relation $F(y, p_{ij}) = 0$ for each $y$.
Explicitly,
\[ F(y, p_{ij}) = \sum_{\substack{I, J \subseteq \{1,\ldots,n\} \\ |I|=|J|}} A_{I,J}H_{I,J} , \]
where $H_{I,J}$ is the $I,J$ sub-minor determinant of the matrix $(p_{ij})$.
This is seen by noting that for each $I, J$,
\[ \pisu{I}\w\omegasu{J} = H_{I,J}\omegasu{\varnothing} + \frac{1}{2}\pDeriv{H_{I,J}}{p_{ij}}(\thetasu{i}\w\omegasu{j} + \thetasu{j}\w\omegasu{i}) + \ldots . \]
so that, pulling $\Upsilon$ up to $\Gr_{n+1}(T M_{-1}, \omegasu{\varnothing})$,
\[ \Upsilon^{(1)} = \sum_{\substack{I, J \subseteq \{1,\ldots,n\} \\ |I|=|J|}} A_{I,J}\pisu{I} \w \omegasu{J} = F \omegasu{\varnothing} + \frac{1}{2}\pDeriv{F}{p_{ij}}(\thetasu{i}\w\omegasu{j} + \thetasu{j}\w\omegasu{i}) + \ldots . \]

Pulling back to $\tilde{M}$, we have
\[ \Upsilon^{(1)} \equiv \frac{1}{2}\pDeriv{F}{p_{ij}}(\thetasu{i}\w\omegasu{j} + \thetasu{j}\w\omegasu{i}) \mod{\thetasu{\varnothing}, \Lambda^2 \I} , \]
where the signature of the matrix $(\pDeriv{F}{p_{ij}})$ determines the symbol of $\tilde{M}$ at each point. Now take $M \subset \tilde{M}$ to be any open set where $(\pDeriv{F}{p_{ij}})$ is positive definite. The exterior differential system $(M,\I)$ is by construction a parabolic system, with parabolic reframings that diagonalize $(\pDeriv{F}{p_{ij}})$. In any parabolic coframing on $M$,
\[ \Upsilon^{(1)} \equiv \thetasu{i}\w\omegasu{i} \mod{\thetasu{\varnothing}, \Lambda^2 \I} . \]
Furthermore, from
\[ \d\Upsilon \equiv 0 \mod{\thetanotu, \d\thetanotu} \]
and commutativity of $\d$ with pullback, it follows that
\[ \d\Upsilon^{(1)} \equiv 0 \mod{\thetasu{\varnothing}, \d\thetasu{\varnothing}, \Upsilon^{(1)}} . \]

It turns out that the existence of such an $\Upsilon^{(1)}$ characterizes the parabolic systems that are the prolongation of a quasi-parabolic system, and that furthermore, the Monge-Amp\`ere invariants characterize the existence of $\Upsilon^{(1)}$. This is summarized in the following Theorem.
\begin{theorem}\label{thm:deprolongation iff invariants vanish iff deprolongation form}
    Given a parabolic system $(M,\I)$ the following are equivalent:
    \begin{enumerate}
        \item $M$ is locally EDS equivalent to a neighborhood in the prolongation of a quasi-parabolic Monge-Amp\`ere system.

        \item There is an $(n+1)$-form $\Upsilon$ on $M$,
        \[ \Upsilon \equiv \thetasu{i} \w \omegasu{i} + \sum_{\substack{I, J \subseteq \{1,\ldots,n\} \\ |I|=|J| > 1}} A_{I,J}\thetasu{I} \w \omegasu{J} \mod{\thetasu{\varnothing}} \]
        such that
        \[ \d\Upsilon \equiv 0 \mod{\thetas{\varnothing}, \d\thetas{\varnothing}, \Upsilon} . \]

        \item The Monge-Amp\`ere invariants of $M$ satisfy $V_i^{0j0} = V_i^{0jk} = 0$ and
        \[ V_i^{jkl} \in \bianchi_2' := \ker(\Sym^2_0 ( \Sym^2_0 \R^n) \xrightarrow{\qquad} \Sym_0^3 \R^n \otimes \R^n) . \]
    \end{enumerate}
\end{theorem}

Before proving the Theorem, it will be useful to establish some algebra.
Let $W = \R^n$ be the standard representation of $\SL(n)$. For each $r \ge 1$, the inclusion
\[ \Lambda^r W \xrightarrow{\qquad} W \otimes \Lambda^{r-1} W \]
induces an equivariant map
\begin{equation}\label{eq: pre Bianchi sequence}
    \Lambda^{r} W \otimes \Lambda^{r} W \xrightarrow{\qquad} \Sym^2 W \otimes \Lambda^{r-1} W \otimes \Lambda^{r-1} W .
\end{equation}

Denote by $\bianchi_r$ the irreducible $\SL(n)$-representation
\[ \ker(\Lambda^{r} W \otimes \Lambda^{r} W \xrightarrow{\qquad} \Lambda^{r+1} W \otimes \Lambda^{r-1} W) . \]
In abstract index notation, an element of $\bianchi_r$ is characterized by a collection of numbers $A_{i_1 j_1 \ldots i_r j_r}$ anti-symmetric under permutation of the $i_1 \ldots i_r$ indices and the $j_1 \ldots i_r$ indices and whose anti-symmetric sum under any $r+1$ indices vanishes.

By restriction of Equation \ref{eq: pre Bianchi sequence}, define the map
\[ \bianchi_r \xrightarrow{\qquad} \Sym^2 W \otimes \bianchi_{r-1} . \]

\begin{lemma}
    For each $r \ge 3$, the sequence
    \[ 0 \xrightarrow{\qquad} \bianchi_r \xrightarrow{\qquad} \Sym^2 W \otimes \bianchi_{r-1} \xrightarrow{\qquad} \Lambda^2 (\Sym^2 W) \otimes \bianchi_{r-2} \]
    is exact.
\end{lemma}
\begin{proof}
    The first map is not zero, so by Schur's Lemma must be injective. Suppose that $A = A_{k_1l_1 k_2l_2 k_3l_3 \ldots k_rl_r} \in \Sym^2 W \otimes \bianchi_{r-1}$, so that $A$ is symmetric in $k_1l_1$, anti-symmetric in $k_2 \ldots k_r$ and $l_2 \ldots l_r$, and satisfies the Bianchi identity in $k_2 l_2 \ldots k_rl_r$. If furthermore $A$ is in the kernel, then
    \[ A_{k_1l_1 k_2l_2 k_3l_3 \ldots k_rl_r} = A_{k_2l_2 k_1l_1 k_3l_3 \ldots k_rl_r} = A_{k_3l_3 k_2l_2 k_1l_1 \ldots k_rl_r} . \]
    But then
    \[ A_{k_1l_1 k_2l_2 k_3l_3 \ldots k_rl_r} = A_{k_3l_3 k_2l_2 k_1l_1 \ldots k_rl_r}  = -A_{k_3l_3 k_1l_2 k_2l_1 \ldots k_rl_r} = -A_{k_2l_1 k_1l_2 k_3l_3 \ldots k_rl_r} \]
    and various permutations shows that $A \in \Lambda^{r} W \otimes \Lambda^{r} W$.

    Finally, it follows quickly from the Littlewood-Richardson rule that
    \[ (\Lambda^{r} W \otimes \Lambda^{r} W) \cap (\Sym^2 W \otimes \bianchi_{r-1}) \cong \bianchi_r , \]
    so the kernel does in fact equal $\bianchi_r$.
\end{proof}

If we branch $W$ along $\SO(n) \subset \SL(n)$, then $\Sym^2 W = \Sym^2_0 W \oplus \R$ and $\Lambda^2 (\Sym^2 W) \cong \Lambda^2(\Sym^2_0 W) \oplus \Sym^2_0 W$. The exact sequence of the Lemma splits accordingly, so we have the following corollary.
\begin{corollary}
    For each $r \ge 3$, the sequence
    \begin{equation}\label{eq:Bianchi sequence}
        0 \xrightarrow{\qquad} \bianchi_r \xrightarrow{\quad f_r \quad} \Sym_0^2 W \otimes \bianchi_{r-1} \xrightarrow{\quad g_r \quad} \Lambda^2 (\Sym_0^2 W) \otimes \bianchi_{r-2}
    \end{equation}
    is exact.
\end{corollary}

It is also straightforward to check that the complex
\begin{equation}\label{eq:Bianchi sequence r=2}
    0 \xrightarrow{\qquad} \bianchi_2 \xrightarrow{\quad f_2 \quad} (\Sym_0^2 W \otimes \Sym^2_0 W)_0 \xrightarrow{\quad g_2 \quad} \Lambda^2 (\Sym_0^2 W)
\end{equation}
is not exact, but the first map is injective. Letting $\bianchi_2'$ denote the image of $\bianchi_2$, the homology at $(\Sym_0^2 W \otimes \Sym^2_0 W)_0$ is isomorphic to $\Sym_0^4 W$, the complement of $\bianchi_2'$ in $\Sym^2_0 (\Sym^2_0 W)$. This homology is directly related to the interpretation of the Monge-Amp\`ere invariants as integrability conditions obstructing a Monge-Amp\`ere deprolongation. Indeed, the secondary Monge-Amp\`ere invariant $(V_i^{jkl})$ takes values in $\Sym_0^2 (\Sym^2_0 W)$, and the following Proposition shows $M$ has a quasi-parabolic deprolongation only if $(V_i^{jkl})$ is zero in homology, that is, if $(V_i^{jkl}) \in \bianchi_2'$. It is useful to note that $\bianchi'_2$ is characterized by
\[ \bianchi'_2 = \ker\left(\Sym^2_0 (\Sym^2_0 W) \xrightarrow{\qquad} \Sym_0^3 W \otimes W\right) . \]

Recall that the Cartan system of $\thetasu{\varnothing}$ is the Frobenius ideal $\cartTheta =  \{\thetasu{\varnothing},\thetasu{a},\omegu^a\}$.
\begin{proposition}\label{thm: (prop) MA iff MA invariants}
    For a parabolic system $(M,\I)$, the Monge-Amp\`ere invariants $V_i^{0j0} = V_i^{0jk} = 0$ and $V_i^{jkl} \in \bianchi_2'$ if and only if there exists a primitive $(n+1)$-form $\Upsilon$ on $M$,
    \[ \Upsilon \equiv \thetasu{i} \w \omegasu{i} + \sum_{\substack{I, J \subseteq \{0,\ldots,n\} \\ |I|=|J| > 1}} A_{I,J}\thetasu{I} \w \omegasu{J} \mod{\thetasu{\varnothing}} \]
    such that
    \[ \d\Upsilon \equiv 0 \mod{\thetasu{\varnothing}, \d\thetanotu, \Upsilon} . \]

    In case this equivalence holds, the form $\thetanotu\w\Upsilon$ is unique and is divisible by $\omegu^0$, so that
    \[ \omegu^0\w\Upsilon \equiv 0 \mod{\thetanotu} . \]
\end{proposition}
\begin{proof}
    Suppose that $V_i^{0j0} = V_i^{0jk} = 0$ and $V_i^{jkl} \in \bianchi_2'$. Then $\Upsilon$ can be constructed as follows. For
    \[ \Upsilon_1 \equiv \thetasu{i}\w\omegasu{i} \mod{\thetanotu} , \]
    we have
    \begin{align*}
        \d\Upsilon_1 & \equiv -2\xi_i^j \thetasu{j}\w\omegasu{i} - \xi_i^0\w(\thetasu{0}\w\omegasu{i} + \thetasu{i}\w\omegasu{0}) \mod{\thetanotu, \d\thetanotu, \Upsilon_1, \Lambda^2 \I} \\
        & \equiv -2V_i^{jkl}\pisu{kl}\w\thetasu{j}\w\omegasu{i} .
    \end{align*}
    The second equivalence relies on the fact that every element of $\Lambda^{n+2} \cartTheta$ is divisible by $\d\thetanotu$.
    Since $V_i^{jkl} \in \bianchi_2'$, there is a unique choice of functions $A_{ij,kl} \in \bianchi_2$ so that $f_2(A_{ij,kl}) = 2V_i^{jkl}$. Setting
    \[ \Upsilon_2 = \Upsilon_1 + \sum_{i,j,k,l \subseteq \{1,\ldots,n\}} A_{ij,kl}\thetasu{i}\w\thetasu{j}\w\omegasu{kl} , \]
    we have, by construction,
    \[ \d\Upsilon_2 \equiv 0 \mod{\thetanotu, \d\thetanotu, \Upsilon_2, \Lambda^2 \I} . \]

    Now, $\Upsilon_2$ is in $\Lambda^{n+1} \cartTheta$, so $\d\Upsilon_2$ is as well. It follows that there are functions $V^{ab}_{I,J}$ so that $\d\Upsilon_2$ has leading part
    \[ \d\Upsilon_2 \equiv \sum_{\substack{I, J \subseteq \{0,\ldots,n\} \\ |I|=|J| = 2}} V_{I,J}^{ab}\pisu{ab}\w\thetasu{I}\w\omegasu{J} \mod{\thetanotu, \Upsilon_2, \Lambda^3 \I} . \]
    Furthermore, since the primary Monge-Amp\`ere invariants vanish, there is a 1-form $\beta_0^0$ on $M$ so that
    \[ \d\omegu^0 \equiv \beta_0^0\w\omegu^0 \mod{\thetanotu, \Lambda^2 \cartTheta} . \]
    Thus from
    \[ \omegu^0\w\Upsilon_2 = 0 , \]
    it follows that
    \[ 0 = d(\omegu^0\w\Upsilon_2) \equiv -\omegu^0\w(\cancel{\beta_0^0\w\Upsilon_2} + \d\Upsilon_2) \mod{\thetanotu, \Lambda^3 \I, \Lambda^{n+3}\cartTheta} . \]
    As a consequence, the leading part of $\d\Upsilon_2$ is divisible by $\omegu^0$, and the functions $V_{I,J}^{ab}$ may be assumed as valued in $(\Sym^2_0 W \oplus W \oplus \R) \otimes \bianchi_2$, so that
    \[ \d\Upsilon_2 \equiv \sum_{\substack{I, J \subseteq \{1,\ldots,n\} \\ |I|=|J| = 2}} V_{I,J}^{ab}\pisu{ab}\w\thetasu{I}\w\omegasu{J} \mod{\thetanotu, \d \thetanotu, \Upsilon_2, \Lambda^3 \I} . \]
    But from $\d^2(\Upsilon_2) = 0$ it follows that $V_{I,J}^{00} = V_{I,J}^{i0} = 0$ and $g_3(V_{I,J}^{ij}) = 0$, and thus by the exactness of Sequence \ref{eq:Bianchi sequence}, there are functions $(A_{I,J}) \in \bianchi_3$ so that $V_{I,J}^{ij} = f_3(A_{I,J})$. Thus we may choose uniquely
    \[ \Upsilon_3 = \Upsilon_2 + \sum_{\substack{I, J \subseteq \{1,\ldots,n\} \\ |I|=|J| = 3}} A_{I,J}\thetasu{I} \w \omegasu{J} \]
    so that
    \[ \d\Upsilon_3 \equiv 0 \mod{\thetanotu, \d \thetanotu, \Upsilon_3, \Lambda^3 \I} . \]
    Continuing inductively, one constructs $\Upsilon$.
    Note that in the construction, each step is uniquely determined, and at each step
    \[ \omegu^0\w\Upsilon_r \equiv 0 \mod{\thetanotu} . \]

    Conversely, suppose given
    \[ \Upsilon \equiv \thetasu{i} \w \omegasu{i} + \sum_{\substack{I, J \subseteq \{0,\ldots,n\} \\ |I|=|J| \ge 2}} A_{I,J}\thetasu{I} \w \omegasu{J} \mod{\thetasu{\varnothing}} \]
    such that
    \[ \d\Upsilon \equiv 0 \mod{\thetasu{\varnothing}, \d\thetanotu, \Upsilon,} . \]
    Adding a multiple of $\thetanotu$ to $\Upsilon$ if necessary, assume that
    \[ \d\Upsilon \equiv 0 \mod{\thetasu{\varnothing}, \Upsilon} . \]
    The leading term of $\Upsilon$ is divisible by $\omegu^0$,
    \[ \omegu^0\w\Upsilon \equiv 0 \mod{\thetasu{\varnothing}, \Lambda^2 \I} , \]
    so
    \[ 0 = \d(\omegu^0\w\Upsilon) \equiv (\d \omegu^0) \w \Upsilon \equiv -V_i^{0ab}\pisu{ab}\w\thetasu{i}\w\omegasu{\varnothing} \mod{\thetasu{\varnothing}, \Lambda^2 \I} , \]
    and thus $V_i^{0j0} = V_i^{0jk} = 0$.

    If $V_i^{0j0} = V_i^{0jk} = 0$ but $V_i^{jkl} \not\in \bianchi_2'$, then there is no choice of $\Upsilon_2$ in the construction given, so $\Upsilon$ cannot exist.
\end{proof}

\begin{proof}[Proof of Theorem \ref{thm:deprolongation iff invariants vanish iff deprolongation form}]
    This Theorem has been proved for systems in $n+1 = 2$ and $n+1=3$ variables by Bryant \& Griffiths \cite{Characteristic_Cohomology_II} and Clelland \cite{Clelland_Thesis} respectively, so I will assume that $n\ge 3$.

    That (1) implies (2) was shown in the construction of the prolongation to a quasi-parabolic Monge-Amp\`ere system. Proposition \ref{thm: (prop) MA iff MA invariants} shows that (2) and (3) are equivalent. Suppose either holds.
    Define on $M$ the exterior ideal
    \[ \K = \{ \thetasu{\varnothing},\ \Upsilon \} . \]
    Condition (2) is exactly the statement that the Cartan system of $\K$ is the Frobenius ideal
    \[ \cartTheta = \{ \thetanotu, \thetasu{a}, \omegu^a \} . \]

    Consider a small enough neighborhood in $M$ so that the space $M_{-1}$ of $\cartTheta$-leaves is a ($2n+3$ dimensional) manifold, which will be the space of the deprolongation.
    Let
    \[ q\colon M \xrightarrow{\quad} M_{-1} \]
    denote the submersion of $M$ onto its leaf space.
    There exists a 1-form $\thetasu{\varnothing}$ and an $(n+1)$-form $\Upsilon$ on $M_{-1}$ so that the ideal
    \[ \I_{-1} = \{\thetasu{\varnothing}, \underline\Upsilon\} \]
    pulls back to generate $\K$.

    On $M$ it is immediate that $\thetas{\varnothing}\w\d\thetas{\varnothing}^{n+1} \neq 0$ and $\Upsilon \not\equiv 0 \mod{\thetanotu, \d\thetanotu}$. But the pullback by $q$ is injective, so the same must hold on $M_{-1}$, and thus $(M_{-1}, \I_{-1})$ is a Monge-Amp\`ere system.

    The primary Monge-Amp\`ere invariants of $M$ vanish and there are $MA$-adapted coframings of $M$ for which (using $n \ge 3$)
    \[ \d\omegu^0 \equiv -\sigma^{0a}\w\thetasu{a} + \xi_i^0\w\omegu^i \equiv 0 \mod{\thetanotu, \omegu^0, \Lambda^2 \cartTheta} . \]
    As such, $\cartTheta$ is the Cartan system of the ideal $\{\thetanotu, \omegu^0\}$, and there is a corresponding rank 2 Pfaffian system on $M_{-1}$ that exhibits $M_{-1}$ as a quasi-parabolic Monge-Amp\`ere system.

    $M$ embeds into the prolongation of $M_{-1}$ as follows. For each point $x \in M$, define the codimension $n+2$ plane $\tilde{E}_x \subset T_x M$ by
    \[ \tilde{E}_x = \{ \thetanotu, \thetasu{a} \}^\perp . \]
    Since $q$ is a submersion, the plane $E_x = q'(\tilde{E}_x) \subset T_{q(x)} M_{-1}$ has dimension $n+1$. Now, $\tilde{E}_x$ is an integral element for $\K$, so $E_x$ is an integral element for $\I_{-1}$. This defines a map $\varphi$ from $M$ to the prolongation $\tilde{M}$ of $M_{-1}$ by $\varphi \colon x \mapsto E_x$. The manifolds $M$ and $\tilde{M}$ have the same dimension, and it follows from the structure equations of $M$ that $\varphi$ is a local diffeomorphism. By construction, $\varphi$ is an EDS equivalence.
\end{proof}

\subsection{Linear Type Systems and Their Symbol}
    In contrast to the general case describe above, there are Monge-Amp\`ere systems whose symbol is well defined independent of choice of $2$-jet of solution.
\begin{definition}
  A Monge-Amp\`ere system $(M_{-1},\I_{-1})$ in $n+1$ variables is of \emph{linear type} if it has an independence condition $J$, locally spanned by $\thetasu{\varnothing}$ and 1-forms $\omegu^0,\ldots,\omegu^n$, so that
  \begin{enumerate}
    \item $J$ is Lagrangian with respect to $\d\thetasu{\varnothing}$:
    \[ \d\thetasu{\varnothing} \equiv 0 \mod{J} . \]
    \item For any 1-form $\underline{\alpha} \in J$,
    \[ \underline{\alpha}\w\Upsilon \equiv 0 \mod{\thetasu{\varnothing},\ \omegu^0\w\ldots\w\omegu^n } . \]
  \end{enumerate}
  $J$ is called a \emph{compatible independence condition} for $(M_{-1},\I_{-1})$.

  A quasi-parabolic Monge-Amp\`ere system is \emph{parabolic} if it is of linear type and $\omegu^0$ is contained in its compatible independence condition.
\end{definition}

Note that linear-type Monge-Amp\`ere systems can arise from non-linearizable differential equations. They correspond to the \nth{2}-order equations that are quasi-linear in second derivatives, i.e. of the form
\[ F(x^a,u,p_a,p_{ab}) = A_\varnothing(x^a,u,p_a) + \sum_{a,b = 0,\ldots n} A_{ab}(x^c,u,p_c) p_{ab} = 0 . \]
It follows that linear equations correspond to linear type Monge-Amp\`ere systems, but not conversely.
In particular, linear type Monge-Amp\`ere equations offer a small, geometrically invariant class containing linear equations.

The nomenclature is in line with linear Pfaffian systems, where `linear' refers to linearity of $\Upsilon$ in the complement of the independence condition. Indeed, by condition 2, there are 1-forms $\etu_{a}$ and a function $L$ so that
\begin{equation}\label{eq:linear MA upsilon}
  \Upsilon \equiv L \omegasu{\varnothing} + \etu_{a}\w\omegasu{a} \mod{\thetasu{\varnothing}} .
\end{equation}
By condition 1 and the maximal non-integrability of $\thetasu{\varnothing}$, the forms $\thetasu{\varnothing}$ and $\omegu^a$ can be completed to a coframing of $M_{-1}$ by forms $\pisu{a}$ so that
\[ \d\thetasu{\varnothing} \equiv - \pisu{a}\w\omegu^a \mod{\thetasu{\varnothing}} . \]
It follows from equation \eqref{eq:linear MA upsilon} and the assumed primitivity of $\Upsilon$ that
\[ 0 \equiv \d\thetasu{\varnothing}\w\Upsilon \equiv (\etu_a\w\pisu{a})\w\omegasu{\varnothing} \mod{\thetasu{\varnothing}} . \]
By an application of Cartan's lemma, there is a symmetric-matrix valued function $(A_{a,b})$ so that
\[ \etu_a \equiv A_{a,b}\pisu{b} \mod{\thetasu{\varnothing}, \omegu^a} , \]
and thus,
\[ \Upsilon \equiv L \omegasu{\varnothing} + A_{a,b}\,\pisu{b}\w\omegasu{a} \mod{\thetasu{\varnothing}} . \]

One can define the $G_{-1}$-structure of coframes adapted as just described, where the structure group $G_{-1}$ consists of matrices of the form
\[ g = \left( \begin{array}{ccc} k_\varnothing & 0 & 0 \\ * & \trans{B^{-1}} & 0 \\ * & BS & B \end{array} \right) \]
for $k_\varnothing \in \R^\times$, $B \in \GL(W)$, and $S \in \Sym^2 W$. Then any two adapted coframings $(\thetasu{\varnothing},\omegu^a,\pisu{a})$ and $(\newcoframe{\thetasu{\varnothing}},\newcoframe{\omegu^a},\newcoframe{\pisu{a}})$ differ by a $G_{-1}$-valued function $g$ on $M$:
\[
\left( \begin{matrix}
\newcoframe{\thetasu{\varnothing}} \\ \newcoframe{\omegu^a} \\ \newcoframe{\pisu{a}}
\end{matrix} \right)
= g^{-1}
\left( \begin{matrix}
\thetasu{\varnothing} \\ \omegu^a \\ \pisu{a}
\end{matrix} \right)
\]

Noting, via the adjugate formula, that
\[ (\newcoframe{\omegasu{a}}) \equiv \det(B)(B^{-1})_a^c (\omegasu{c}) \mod{\thetasu{\varnothing}} , \]
it is not difficult to check that the function $A = (A_{a,b})$, lifted to this $G_{-1}$-structure, varies by the rule
\[ (A_{a,b}(g\cdot u)) = \det(B^{-1}) B^c_a A_{c,d} B^d_b \]
Consequently, there are choices of adapted coframe that normalize $A_{a,b}$, depending on the signature of $A$.

In the case that $M_{-1}$ is a parabolic Monge-Amp\`ere system, the unadapted symbol $A_{a,b}$ is positive semi-definite everywhere, and there are reduced coframings of extending $\thetanotu, \omegu^0$,
\begin{equation}\label{eq:MA parabolic coframing}
  \thetasu{\varnothing}, \omegu^a, \pisu{a}
\end{equation}
so that
\[ \d\thetasu{\varnothing} \equiv -\pisu{a}\w\omegu^a \mod{\thetasu{\varnothing}} \]
and
\[ \Upsilon \equiv L \omegasu{\varnothing} + \pisu{i}\w\omegasu{i} \mod{\thetasu{\varnothing}} . \]

The following Theorem characterizes the parabolic systems that can be de-prolonged to a parabolic Monge-Amp\`ere system.
\begin{theorem}
    Given a parabolic system $(M,\I)$ the following are equivalent:
    \begin{enumerate}
        \item $M$ is locally EDS equivalent to a neighborhood in the prolongation of a parabolic Monge-Amp\`ere system.

        \item The $(n+1)$-form
        \[ \Upsilon \equiv \sum_{i=1}^n \theta_i \w \omega_{(i)} \mod{\thetas{\varnothing}} \]
        on $M$ satisfies
        \[ \d\Upsilon \equiv 0 \mod{\thetas{\varnothing}, \d\thetas{\varnothing}, \Upsilon} . \]

        \item The primary, secondary and tertiary Monge-Amp\`ere invariants of $M$ vanish identically.
    \end{enumerate}
\end{theorem}
\begin{proof}
    If $M$ embeds in the prolongation of a parabolic Monge-Amp\`ere system $(M_{-1}, \I_{-1} = \{\thetanotu, \Upsilon'\})$, then the pullback of $\Upsilon'$ to $M$ will be an $n+1$-form $\Upsilon$ satisfying condition (2).

    If $M$ satisfies condition (2), then from
    \[ 0 \equiv \d\Upsilon \equiv -2\xi_i^j \w \thetasu{j}\w\omegasu{i} - \xi_i^0\w(\thetasu{0}\w\omegasu{i} + \thetasu{i}\w\omegasu{0}) \mod{\thetanotu, \d \thetanotu, \Upsilon, \Lambda^2\I} \]
    one sees that the primary and secondary Monge-Amp\`ere invariants vanish. Then, assuming that a $MA$-adapted coframing has been chosen,
    \[ \d\Upsilon \equiv V\pisu{jk}\w\thetasu{i}\w\thetasu{k}\w\omegasu{ji} \mod{\thetanotu, \d \thetanotu, \Upsilon} , \]
    so the tertiary Monge-Amp\`ere invariant must vanish as well.

    This calculation also shows that if the primary, secondary, and tertiary Monge-Amp\`ere invariants vanish on $M$, then $\Upsilon$ satisfies condition 2.

    Finally, if conditions (2) and (3) hold, then from Theorem \ref{thm:deprolongation iff invariants vanish iff deprolongation form}, $M$ embeds in the prolongation of a quasi-parabolic system $M_{-1}$.  But now $\sigma^{ab} \in \cartTheta$, so
    \[ \d\omegu^a \equiv 0 \mod{\thetanotu, \omegu^b, \Lambda^2 \cartTheta} . \]
    It follows that
    \[ \widetilde{J} = \{\thetasu{\varnothing}, \omegu^a\} \]
    has Cartan system $\cartTheta$, so descends to define an ideal
    \[ J = \{\thetanotu, \omegu^a \} \]
    on $M_{-1}$.

    On $M$, the conditions
    \[ \d\tilde\theta_{\varnothing} \equiv 0 \mod{\widetilde{J}} \]
    and for any $\alpha \in \widetilde{J}$,
    \[ \alpha\w\Upsilon \equiv 0 \mod{\tilde\theta_{\varnothing}, \tilde\omega^0\w\ldots\w\tilde\omega^n} \]
    are obvious. But the map $M \to M_{-1}$ is a submersion, so it follows that $M_{-1}$ is a linear type.

    On $M$, it is clear that $\{\thetanotu, \omegu^0\} \subset \{\thetanotu, \omegu^a\}$, so in fact $M_{-1}$ is a parabolic Monge-Amp\`ere equation.
\end{proof}

\begin{example}
  It is not difficult to check from the coframing given in Example \ref{ex:Mean Curvature Flow} that the Monge-Amp\`ere invariants of the mean curvature flow vanish.
\end{example}

\section{The Extended Goursat Invariant}\label{sec:Goursat invts}
\subsection{Definition of the Goursat Invariants}
Continuing on the equivalence problem, assume that the primary Monge-Amp\`ere invariants of a parabolic system $M$ vanish. Given a 1-adapted coframing, there are functions $a_{ij}$ and $a_i$ on $\gBun_1$ so that
\[ \xi_i^0 \equiv a_{ij}\omega^j \mod{\thetas{\varnothing},\thetas{a}, \omega^0} . \]

In this case,
\begin{align*}
  0 = \d^2\thetas{i} & = d(-\kappa_i\w\thetas{\varnothing} -\beta_i^j\w\thetas{j} - \pis{ij}\w\omega^j - \pis{i0}\w\omega^0 - \xi_i^0\w\thetas{0} - \xi_i^j\w\thetas{j}) \\
   & \equiv (-\beta_i^j\w\xi^0_j - \gamma^0_{ij}\w\omega^j - \d\xi^0_i + \beta_0^0\w\xi^0_i)\w\thetas{0} \mod{\thetas{\varnothing}, \thetas{i}, \omega^0, \sbforms^3}
\end{align*}
gives, after an application of Cartan's lemma,
\[ \d\xi_i^0 \equiv \beta_0^0\w\xi_i^0 - \beta_i^j\w\xi_j^0 - \gamma^0_{ij}\w\omega^j \mod{\thetas{\varnothing},\thetas{a},\omega^0,\sbforms^2} . \]
Plugging in for $\xi_i^0$, one finds that
\[ (\d a_{ij} - a_{ik}(\delta^k_j\kappa_\varnothing - \beta^k_j))\w\omega^j \equiv (\beta_0^0a_{ij} - \beta_i^k a_{kj} - \gamma^0_{ij})\w\omega^j \mod{\thetas{\varnothing},\thetas{a},\omega^0,\sbforms^2} \]
and thus
\begin{align*}
    \d a_{ij} & \equiv (\beta_0^0 + \kappa_\varnothing)a_{ij} - \beta_i^k a_{kj} - a_{ik}\beta_j^k - \gamma^0_{ij} && \mod{\sbforms^1}
\end{align*}
Integrating, the functions $a_{ij}$ vary in the fiber of $\gBun_1$ by the rule
\begin{equation}\label{eq:Goursat variation}
  (a_{ij}(g\cdot u)) = \frac{B_0^0 k_\varnothing}{b^2}(\mathring{B}^{-1})_i^k(a_{kl}(u))(\mathring{B}^{-1})_j^l - D^0_{ij}
\end{equation}
for all $u \in \gBun_1$ and $g\in G_1$.
The component $(D^0_{ij})$ of $g$ is symmetric and traceless, so there are 1-adapted coframes for which the symmetric, traceless component of $a_{ij}$ vanishes.
The remnant splits into a pure trace component and an antisymmetric component,
\[ a_{ij} = a\delta_{ij} + \hat{a}_{ij} . \]

Now recall the normalization $\pisu{ii} = 0$ that was made to define $0$-adapted coframings. If this requirement is relaxed, then it is possible to absorb the trace component of $a_{ij}$. Indeed, the parabolic change of coframing
\[ \pisu{ij} \xmapsto{\qquad}  \pisu{ij} + a\delta_{ij}\thetasu{0} \]
absorbs $a$, at cost of
\[ \pisu{ii} = n a \thetasu{0} . \]

Say that a parabolic coframing of $M$ is $2$-adapted if $V_i^{0ja} = 0$ and the symmetric component of $a_{ij}$ is absorbed, so that $a_{ij} = \hat{a}_{ij}$.
These adapted coframes form a $G_{2}$-structure\footnote{This group is unrelated to the exceptional simple group $G_2$.} $\gBun_2$ on $M$, where $G_2$ consists of matrices as in \eqref{eq:G_0 structure matrix} for which \[ S^{0a} = 0 \mbox{\quad and\quad } D^0_{ij} = 0 . \]
The pseudo-connection forms $\gamma^0_{ij}$ reduce to semi-basic torsion forms on $\gBun_2$.

The function $a$ is related to the classical Goursat invariant (\cite{Goursat:SurLIntegration}, but see \cite{Characteristic_Cohomology_II} for historical discussion) and measures whether the subprincipal symbol vanishes. The component $\hat{a}_{ij}$, called the \emph{extended Goursat invariant},  measures whether there is a choice of coordinates for which the parabolic equation is in evolutionary form.

Consider first the classical Goursat invariant. The variation of $a$ is given by Equation \ref{eq:Goursat variation}, which simplifies to
\begin{equation*}
    a(g\cdot u) = \frac{B_0^0 k_\varnothing}{b^2}a .
\end{equation*}
As such, in any neighborhood where $a$ is non-zero, there are $2$-adapted coframes for which $a = 1/n$, and thus $\pisu{ii} = \thetasu{0}$.
\begin{definition}
    A parabolic system $M$ with $2$-adapted coframings is \emph{Goursat} if $\pisu{ii} = 0$ everywhere.

    $M$ is \emph{non-Goursat} if $a$ is nowhere vanishing, in which case the 2-adapted coframings for which $\pisu{ii} = \thetasu{0}$ are said to be \emph{3-adapted}.
\end{definition}
The bundle $\gBun_3$ of $3$-adapted coframes has structure group $G_3$, consisting of the matrices in $G_2$ for which $B_0^0 = b^2/k_{\varnothing}$.

Turning to the extended Goursat invariant, we have the following Lemma.
\begin{lemma}
    A parabolic system $(M,\I)$
    has 3-adapted coframings for which furthermore
    \[ \d\omegu^0 \equiv 0 \mod{\omegu^0} \]
    if and only if
    its primary Monge-Amp\`ere invariants $(V^{0ab}_i)$ and the extended Goursat invariants $(\goursam_{ij})$ vanish identically.
\end{lemma}

\begin{proof}
    If $(M,\I)$ has a 3-adapted coframing with $\d\omegu^0\w\omegu^0 = 0$, then
    \[ 0 \equiv \d\omegu^0 \equiv \xi_i^0\w\omegu^i \equiv (V^{0ab}_i\pisu{ab} + \hat{a}_{ij}\omegu^j)\w\omegu^i \mod{\thetasu{\varnothing}, \thetasu{a}, \omegu^0} , \]
    so that $V^{0ab}_i = 0$ and then $\goursam_{ij} = 0$.

    Conversely, suppose that both $(V^{0ab}_i)$ and $(\goursam_{ij})$ are identically zero on $M$ and fix a 3-adapted coframing. Then there is a matrix valued function $(N_i^a)$ and an anti-symmetric matrix valued function $(M^{ab})$ so that
    \[ \d\omegu^0 \equiv -\kappa^0\w\thetasu{\varnothing} + M^{ab}\thetasu{a}\w\thetasu{b} + N_i^a\thetasu{a}\w\omegu^i \mod{\omegu^0} . \]
    From
    \[ 0 = \d^2 \omegu^0 \equiv -N_i^a \pisu{aj}\w\omegu^j\w\omegu^i \mod{\thetasu{\varnothing},\thetasu{a},\omegu^0} . \]
    Collecting coefficients for the cases $a\neq i$ and $a = i$,
    \[ 0 \equiv - \sum_{a \neq i} (N_i^a\omegu^i)\w\pisu{aj}\w\omegu^j - \sum_{i \neq j} (N_i^i - N_j^j) \pisu{ij}\w\omegu^j\w\omegu^i \mod{\thetasu{\varnothing},\thetasu{a},\omegu^0} , \]
    it follows that
    \[ N^0_i = 0 \quad \mbox{and} \quad N_i^k = N_{\tr} \delta_i^k \]
    for some function $N_{\tr}$ on $M$.
    Reduce coframes by replacing $\omegu^0$ with $\omegu^0 + N_{\tr}\thetasu{\varnothing}$, giving a new $3$-adapted coframing for which
    \[ \d\omegu^0 \equiv -\kappa^0\w\thetasu{\varnothing} + M^{ab}\thetasu{a}\w\thetasu{b} \mod{\omegu^0} . \]

    Continuing with this reduction,
    \[ 0 = \d^2 \omegu^0 \equiv M^{ab}(\pisu{ai}\w\thetasu{b}\w\omegu^i - \pisu{bi}\w\thetasu{a}\w\omegu^i) -\kappa^0\w\thetasu{i}\w\omegu^i \mod{\thetasu{\varnothing},\omegu^0,\Lambda^2 \I} . \]
    For each $a, b$ and choice of $i$ equal to neither, the coefficient of $\pisu{ai}\w\thetasu{b}\w\omegu^i$ is $M^{ab}$, so
    \[ (M^{ab}) = 0 . \]
    Finally, since
    \[ 0 = \d^2 \omegu^0 \equiv \kappa^0\w\thetasu{i}\w\omegu^i \mod{\thetasu{\varnothing},\omegu^0} , \]
    it follows that
    \[ \kappa^0 \equiv 0 \mod{\thetasu{\varnothing},\omegu^0} , \]
    and thus
    \[ \d \omegu^0 \equiv 0 \mod{\omegu^0} . \]
\end{proof}

%
The set of coframes with $N_{\tr}$ absorbed is a $G_4$-structure $\gBun_4$, where $G_4$ consists of the matrices in $G_3$ for which
\[ k^0 = 0 . \]
The pseudo-connection form $\kappa^0$ is semi-basic when restricted to $\gBun_4$.

\subsection{Geometric interpretation of the Goursat Invariants}\label{sec:Integrable characteristics}

The following theorem shows that the Goursat invariant are exactly the measure of whether a parabolic equation can be put into evolutionary form.
\begin{theorem}\label{thm:integrable characteristics}
  For a real analytic parabolic system $(M,\I)$ in $n+1>3$ variables, the following conditions are equivalent:
  \begin{enumerate}
    \item The primary Monge-Amp\`ere invariants $(V^{0ab}_i)$ and the extended Goursat invariants $\goursam_{ij}$ vanish identically while the classical Goursat invariant is nowhere vanishing.
    \item $M$ has a $3$-adapted coframing so that
    \[ \d\omegu^0 \equiv 0 \mod{\thetasu{\varnothing},\thetasu{a},\omegu^0} \]
    and
    \[ \pisu{ii} = \thetasu{0} . \]
    \item $M$ has a $4$-adapted coframing so that
    \[ \d\omegu^0 \equiv 0 \mod{\omegu^0} \]
    and
    \[ \pisu{ii} = \thetasu{0} . \]
    \item $M$ is locally equivalent to a parabolic equation in evolutionary form.
  \end{enumerate}
\end{theorem}
\begin{proof}
  Conditions (1), (2), and (3) are equivalent by the work of the previous sections. In particular, the coframing given in condition (2) can be immediately reduced to a $4$-adapted coframing.

  If $M$ is locally equivalent to an evolutionary parabolic
  \[ p_0 = F(x^a,u,p_i,p_{ij}) , \]
  then the $0$-adapted coframings of $M$ as described in Example \ref{ex:jet formulation of parabolic} can readily be $4$-adapted. Indeed, using the notation there, at any point of $M$, the relation given by $\d F = 0$ reduces to
  \[ \hat\theta_{0} \equiv \d p_0 = \d F \equiv \pDeriv{F}{p_{ij}}\d p_{ij} \mod{\thetasu{\varnothing}, \thetasu{i}, \omegu^a} , \]
  and any choice of $B_i^j$ diagonalizing the symbol will preserve $\hat\omega^0 = \d x^0$.

  That (3) implies (4) will be shown by constructing a local EDS embedding of $M$ into $J^2(\R^{n+1},\R)$, explicitly exhibiting $M$ as locally equivalent to a parabolic equation in evolutionary form.

  Fix a $4$-adapted coframing
  \[ \thetasu{\varnothing}, \thetasu{a}, \omegu^a, \pisu{ab} \]
  near a point $y \in M$. By rescaling $\omegu^0$ if necessary, there is a function $x^0$ defined near $y$ so that
  \[ \omegu^0 = \d x^0 . \]

  Since
  \[ \thetasu{\varnothing}\w\d\thetasu{\varnothing}^{n+1}\w\d x^0 = 0 , \qquad \thetasu{\varnothing}\w\d\thetasu{\varnothing}^{n}\w\d x^0 \neq 0,  \]
  a relative version of Darboux's Theorem shows that there are independent functions $x^i, u, p_a$ on $M$ so that (rescaling $\thetasu{\varnothing}$ if necessary)
  \[ \thetasu{\varnothing} = \d u - p_i \d x^i - p_0 \d x^0 . \]
  Since
  \begin{align*}
    \d \thetasu{\varnothing} \equiv & -\d p_i\w\d x^i - \d p_0\w\d x^0 \mod{\thetasu{\varnothing}} \\
     \equiv & - \thetasu{i}\w\omegu^i - \thetasu{0}\w\d x^0 \mod{\thetasu{\varnothing}} ,
  \end{align*}
  there is a symplectic-matrix valued function, with block components of the form
  \[ \left(\begin{array}{cc}
     A & B \\
     C & D
  \end{array}\right) \qquad
  \begin{array}{c}
      \trans{A}C, \ \trans{B}D \mbox{ symmetric, } \\
      \trans{A}D - \trans{C}B = \Id
  \end{array} \]
  so that
  \begin{align*}
    \thetasu{a} \equiv & A^b_a\d p_b + B_{ab}\d x^b \mod{\thetasu{\varnothing}} \\
    \omegu^a \equiv & C^{ab}\d p_b + D^a_b\d x^b \mod{\thetasu{\varnothing}} .
  \end{align*}

  By adding $p_ix^i$ (no sum) to $u$, any $x^i$ can be exchanged with the corresponding $p_i$ if necessary, so assume without loss of generality that $A$ is of the form
  \[ A = \left( \begin{array}{cc}
    A_0^0 & A^i_0 \\
    A^0_i & A'
  \end{array} \right) \in
  \left( \begin{array}{cc}
    \R & \trans{\R^n} \\
    \R^n & \GL(\R^n)
  \end{array} \right) . \]
  It is clear from $\omegu^0 = \d x^0$ that
  \[ C^{0a} = 0, \qquad D^0_i = 0, \quad \mbox{ and } \quad D^0_0 = 1 , \]
  so write
  \[ C =
  \left( \begin{array}{cc}
    0 & 0 \\
    C^{i0} & C'
  \end{array} \right) \in
  \left( \begin{array}{cc}
    \R & \trans{\R^n} \\
    \R^n & \operatorname{End}(\R^n)
  \end{array} \right). \]

  Since $\trans{A}C$ is symmetric, so is $\trans{A'}C'$, as well as $C'(A')^{-1}$. After the parabolic change of coframing
  \[  \omegu^i \xmapsto{\quad} \omegu^i + (C'(A')^{-1})^{ij} \thetasu{j} , \]
  we may assume without loss of generality that $C' = 0$. But once $C' = 0$, it follows from the symmetry of $\trans{A}C$ plus the invertibility of $A'$ that $C = 0$. Then, from $\trans{A}D - \trans{C}B = \Id$ it follows that
  \[ D = \trans{A^{-1}} . \]
  Since each previous step preserved $\omegu^0 = \d x^0$, it must also be true that $(A^0_i) = 0$ and $A^0_0 = 1$. Note that this means that
  \[ \thetasu{0} \equiv \d p_0 \mod{\omegu^a} . \]

  Let
  \[ p_{ab} = -(A^{-1})_a^cB_{cb} , \]
  which is symmetric in the indices $a, b$ because $\trans{B}D = \trans{B}\trans{A}^{-1}$ is symmetric. Then
  \begin{align*}
    \thetasu{a} \equiv & A^b_a(\d p_b - p_{bc}\d x^c) \mod{\thetasu{\varnothing}} \\
    \omegu^a \equiv & (\trans{A}^{-1})^a_b\d x^b \mod{\thetasu{\varnothing}} .
  \end{align*}

  Now, it is straightforward to compute that
  \[ -\pisu{ia}\w\omegu^a \equiv \d\thetasu{i} \equiv -(A_i^b \d p_{bc} A^c_a)\w\omegu^a \mod{\thetasu{\varnothing}, \thetasu{j}} , \]
  \[ -\pisu{0a}\w\omegu^a \equiv \d\thetasu{0} \equiv -(A_0^b \d p_{bc} A^c_a)\w\omegu^a \mod{\thetasu{\varnothing}, \thetasu{a}} , \]
  from which it follows that
  \begin{align*}
      \pisu{ij} & \equiv A_i^a \d p_{ab} A^b_j = A_i^k \d p_{kl} A^l_j && \mod{\thetasu{\varnothing}, \thetasu{j}, \omegu^a} , \\
      \pisu{i0} & \equiv A_i^a \d p_{ab} A^b_0 = A_i^j \d p_{jb} A^b_0 && \mod{\thetasu{\varnothing}, \thetasu{a}, \omegu^a} , \\
      \pisu{00} & \equiv A_0^a \d p_{ab} A^b_0 && \mod{\thetasu{\varnothing}, \thetasu{a}, \omegu^a} .
  \end{align*}
  These can be solved for the $\d p_{ab}$:
  \begin{align*}
      \d p_{ij} & \equiv (A^{-1})_i^k \pisu{kl} (A^{-1})^l_j && \mod{\thetasu{\varnothing}, \thetasu{j}, \omegu^a } , \\
      \d p_{i0} & \equiv (A^{-1})_i^j \pisu{j0} && \mod{\thetasu{\varnothing}, \thetasu{a}, \omegu^a, \pisu{jk} } , \\
      \d p_{00} & \equiv \pisu{00} && \mod{\thetasu{\varnothing}, \thetasu{a}, \omegu^a, \pisu{jk}, \pisu{j0}} .
  \end{align*}

  In particular, the exact 1-forms $\d x^a, \d u, \d p_a, \d p_{ab}$ span the cotangent bundle of $M$ near $y$. Consequently, the functions $x^a, u, p_a, p_{ab}$ define an embedding of a neighborhood of $y$ into $J^2(\R^{n+1},\R)$.
  Identifying the functions $x^a, u, p_a, p_{ab}$ with the corresponding jet coordinates via pullback, it is straightforward that the embedding is an EDS morphism. By simple dimension count, there is a single relation
  \[ F(x^a, u, p_a, p_{ab}) = 0 \]
  whose zero locus is the image of $M$. In other words, $(M,\I)$ is locally defined near $y$ by the differential equation $F$.

  It remains to show that $F$ defines an evolutionary parabolic equation.
  Consider the pullback of $\d F$ to $M$,
  \[ 0 = \d F \equiv \pDeriv{F}{p_{0}} \d p_{0} + \pDeriv{F}{p_{ij}} \d p_{ij} + \pDeriv{F}{p_{i0}} \d p_{i0} + \pDeriv{F}{p_{00}} \d p_{00} \mod{\thetasu{\varnothing},\thetasu{i}, \omegu^a} \]
  First, note that the forms $\pisu{0a}$ are independent of the rest of the coframe forms, so
  \[ 0 = \d F \equiv \pDeriv{F}{p_{00}}\pisu{00} \mod{\thetasu{\varnothing},\thetasu{a}, \omegu^a, \pisu{aj}} \]
  gives $\pDeriv{F}{p_{00}} = 0$ and then
  \[ 0 = \d F \equiv \pDeriv{F}{p_{i0}}(A'^{-1})_i^j \pisu{j0} \mod{\thetasu{\varnothing},\thetasu{a}, \omegu^a, \pisu{jk}} \]
  gives
  \[ \pDeriv{F}{p_{a0}} = 0 . \]
  Consequently,
  \[ F = F(x^a, u, p_a, p_{ij}) . \]

  Since there is exactly one linear relation between the forms of the coframing, there must be a non-vanishing function $\lambda$ so that
  \begin{align*}
      \lambda(\thetasu{0} - \pisu{ii}) \equiv \d F & \equiv \pDeriv{F}{p_{0}} \d p_{0} + \pDeriv{F}{p_{ij}} \d p_{ij} && \mod{\thetasu{\varnothing},\thetasu{i}, \omegu^a} \\
      & \equiv \pDeriv{F}{p_{0}} \thetasu{0} + \pDeriv{F}{p_{ij}} (A^{-1})_i^k \pisu{kl} (A^{-1})^l_j &&
  \end{align*}
  It follows that
  \[ \lambda\thetasu{0} \equiv \pDeriv{F}{p_{0}}\thetasu{0} \mod{\thetasu{\varnothing},\thetasu{i}, \omegu^a, \pisu{ij}} , \]
  so that $\pDeriv{F}{p_{0}}$ is non-vanishing in a neighborhood of $y$. Then, by the implicit function theorem, $F$ is equivalent to an evolutionary equation
  \[ p_0 = G(x^a, u, p_i, p_{ij}) \]
  in a possibly smaller neighborhood of $y$.
\end{proof}

With only slight modification to the previous proof, one may show the following for Goursat parabolic systems.
\begin{proposition}
    For a real analytic parabolic system $(M,\I)$ in $n+1>3$ variables, the primary Monge-Amp\`ere invariants $(V^{0ab}_i)$ as well as classical and extended Goursat invariants vanish identically if and only if $M$ is locally equivalent to a sub-elliptic equation of the form
    \[ F(x^a, u , p_i, p_{ij}) = 0 . \]
\end{proposition}
\begin{proof}
    The only change that need be made from the previous proof is to note that in the Goursat case, $\pisu{ii} = 0$ on $M$, so there must be a non-vanishing function $\lambda$ such that
    \[ \lambda\pisu{ii} \equiv \d F \equiv \pDeriv{F}{p_{0}} \d p_{0} + \pDeriv{F}{p_{ij}} \d p_{ij} \mod{\thetasu{\varnothing},\thetasu{i}, \omegu^a} \]
    Thus
    \[ 0 \equiv \d F \equiv \pDeriv{F}{p_{0}} \d p_{0} \equiv \pDeriv{F}{p_{0}}\thetasu{0} \mod{\thetasu{\varnothing},\thetasu{i}, \omegu^a, \pisu{ij}} , \]
    and it follows that $\pDeriv{F}{p_{0}}$ is identically zero.
\end{proof}

\section{Further Directions}
This paper does not fully resolve the equivalence problem for parabolic equations, and there are certainly deeper invariants. As is typical in equivalence problems, the number of posibilities grows into exponentially many cases, and one should not expect to explicate all in a reasonable amount of time. Instead, one may hope to draw out the most geometrically interesting invariants.

Likely the most interesting branch of the equivalence problem to follow would be the Monge-Amp\`ere systems. In particular, it seems that the case of linear type Monge-Amp\`ere equations will turn out to be worthy of further study. The notion of linear type can be defined for \nth{2}-order equations of arbitrary symbol type, and indeed, are this is the class of Monge-Amp\`ere equations where symbol type is most simply defined.

One reason to focus on Monge-Amp\`ere equations is that they are good candidates for non-linear equations with large, even infinite, families of conservation laws. Indeed, the results of the follow up paper to this one (\cite{McMillan:ParabolicsII}) may be interpreted as evidence for this supposition, for it is shown there that any evolutionary parabolic equation with even a single conservation is Monge-Amp\`ere.

\bibliographystyle{amsplain}
\bibliography{references}

\begin{thebibliography}{1}

\bibitem{Goursat:SurLIntegration}
Le{\c{c}}ons sur l'integration des {\'e}quations aux d{\'e}riv{\'e}es
  partielles du second ordre a deux variables ind{\'e}pendantes.
\newblock {\em Monatshefte f{\"u}r Mathematik und Physik}, 9(1):A29--A30, Dec
  1898.

\bibitem{BCGGG}
R.L. Bryant, S.S. Chern, R.B. Gardner, H.L. Goldschmidt, and P.A. Griffiths.
\newblock {\em Exterior Differential Systems}.
\newblock Mathematical Sciences Research Institute Publications. Springer New
  York, 2013.

\bibitem{BryantGriffithsGrossman:EDSandEulerLagrangePDEs}
Robert Bryant, Phillip Griffiths, and Daniel Grossman.
\newblock {\em Exterior differential systems and {E}uler-{L}agrange partial
  differential equations}.
\newblock Chicago Lectures in Mathematics. University of Chicago Press,
  Chicago, IL, 2003.

\bibitem{BryantGriffithsHsu:TowardAGeometryofDifferentialEquations}
Robert Bryant, Phillip Griffiths, and Lucas Hsu.
\newblock Toward a geometry of differential equations.
\newblock In {\em Geometry, topology, \& physics}, Conf. Proc. Lecture Notes
  Geom. Topology, IV, pages 1--76. Int. Press, Cambridge, MA, 1995.

\bibitem{Characteristic_Cohomology_II}
Robert~L. Bryant and Phillip~A. Griffiths.
\newblock Characteristic cohomology of differential systems {II}: Conservation
  laws for a class of parabolic equations.
\newblock {\em Duke Math. J.}, 78(3):531--676, 06 1995.

\bibitem{Gardner:MethodOfEquivalence}
R.~Gardner.
\newblock {\em The Method of Equivalence and Its Applications}.
\newblock Society for Industrial and Applied Mathematics, 1989.

\bibitem{McMillan:ParabolicsII}
Benjamin~B. {McMillan}.
\newblock {Geometry and Conservation Laws for a Class of Second-Order Parabolic
  Equations II: Conservation Laws}.
\newblock {\em ArXiv e-prints}, page arXiv:1810.02346, October 2018.

\bibitem{Clelland_Thesis}
Jeanne N.~Clelland.
\newblock Geometry of conservation laws for a class of parabolic partial
  differential equations.
\newblock {\em Selecta Mathematica}, 3(1):1--77, 1997.

\end{thebibliography}
\end{document}